\documentclass[11pt,reqno]{amsart}
\usepackage[top = 1in, bottom = 1in, left=1.0in, right=1.0in]{geometry}

\usepackage{amssymb,amsmath,amsthm,graphicx,xcolor,mathtools,enumerate,mathrsfs,tabularx,bbm,tikz,url}
\usepackage[shortlabels]{enumitem} 
\usepackage[british]{babel}
\usepackage[utf8]{inputenc}
\allowdisplaybreaks

\usepackage{hyperref}
\hypersetup{colorlinks, linkcolor={red!50!black}, citecolor={green!50!black}, urlcolor={blue!50!black}}

\usepackage{caption}
\usepackage{subcaption}

\usepackage[mathlines]{lineno}
\usepackage{etoolbox} 

\newcommand*\linenomathpatch[1]{%
\expandafter\pretocmd\csname #1\endcsname {\linenomath}{}{}%
\expandafter\pretocmd\csname #1*\endcsname{\linenomath}{}{}%
\expandafter\apptocmd\csname end#1\endcsname {\endlinenomath}{}{}%
\expandafter\apptocmd\csname end#1*\endcsname{\endlinenomath}{}{}%
}
\newcommand*\linenomathpatchAMS[1]{%
\expandafter\pretocmd\csname #1\endcsname {\linenomathAMS}{}{}%
\expandafter\pretocmd\csname #1*\endcsname{\linenomathAMS}{}{}%
\expandafter\apptocmd\csname end#1\endcsname {\endlinenomath}{}{}%
\expandafter\apptocmd\csname end#1*\endcsname{\endlinenomath}{}{}%
}

\expandafter\ifx\linenomath\linenomathWithnumbers
\let\linenomathAMS\linenomathWithnumbers
\patchcmd\linenomathAMS{\advance\postdisplaypenalty\linenopenalty}{}{}{}
\else
\let\linenomathAMS\linenomathNonumbers
\fi

\linenomathpatchAMS{gather}
\linenomathpatchAMS{multline}
\linenomathpatchAMS{align}
\linenomathpatchAMS{alignat}
\linenomathpatchAMS{flalign}
\linenomathpatch{equation}

\usepackage{cleveref}

\theoremstyle{plain}
\newtheorem{theorem}{Theorem}[section]
\crefname{theorem}{Theorem}{Theorems}

\newtheorem{proposition}[theorem]{Proposition}
\crefname{proposition}{Proposition}{Propositions}

\newtheorem{corollary}[theorem]{Corollary}
\crefname{corollary}{Corollary}{Corollaries}

\newtheorem{lemma}[theorem]{Lemma}
\crefname{lemma}{Lemma}{Lemmas}

\crefname{conjecture}{Conjecture}{Conjectures}

\crefname{problem}{Problem}{Problem}

\newtheorem{claim}[theorem]{Claim}
\crefname{claim}{Claim}{Claims}

\crefname{observation}{Observation}{Observations}

\crefname{setup}{Setup}{Setups}

\crefname{fact}{Fact}{Facts}

\crefname{algorithm}{Algorithm}{Algorithms}

\newtheorem{remark}[theorem]{Remark}
\crefname{remark}{Remark}{Remarks}

\newtheorem{example}[theorem]{Example}
\crefname{example}{Example}{Examples}

\theoremstyle{definition}
\newtheorem{definition}[theorem]{Definition}
\crefname{definition}{Definition}{Definitions}

\crefname{construction}{Construction}{Constructions}

\crefname{question}{Question}{Questions}

\numberwithin{equation}{section}

\crefname{section}{Section}{Sections}
\crefname{appendix}{Appendix}{Appendix}

\newenvironment{proofclaim}[1][Proof of the claim]{\begin{proof}[#1]}{\end{proof}}

\def\COMMENT#1{}

\usepackage{comment}

\DeclareMathOperator{\probability}{\mathbb{P}}

\newcommand{\eps}{\varepsilon}
\renewcommand{\rho}{\varrho}

\renewcommand{\subset}{\subseteq}

\newcommand{\NATS}{\mathbb{N}}

\newcommand{\cC}{\mathcal{C}}

\newcommand{\cE}{\mathcal{E}}
\newcommand{\cF}{\mathcal{F}}
\newcommand{\cG}{\mathcal{G}}
\newcommand{\cH}{\mathcal{H}}

\newcommand{\cP}{\mathcal{P}}
\newcommand{\cQ}{\mathcal{Q}}
\newcommand{\cR}{\mathcal{R}}

\newcommand{\cV}{\mathcal{V}}

\newcommand{\cX}{\mathcal{X}}

\DeclareMathOperator{\hf}{\mathsf{HF}}
\newcommand{\bN}{\mathbb{N}}

\newcommand{\sP}{\mathsf{P}}

\newcommand{\ori}[1]{\smash{\overrightarrow{#1}}}

\DeclareMathOperator{\Deg}{\mathsf{Deg}}

\DeclareMathOperator{\oG}{G}

\DeclareMathOperator{\til}{\mathsf{Til}} 
\DeclareMathOperator{\mat}{\mathsf{Mat}}
\DeclareMathOperator{\ham}{\mathsf{Ham}} 
\DeclareMathOperator{\powham}{\mathsf{PowHam}}

\DeclareMathOperator{\den}{\mathsf{Den}}

\let\th\relax
\DeclareMathOperator{\th}{\delta}

\DeclareMathOperator{\HamConn}{\mathsf{HamCon}} 
\DeclareMathOperator{\HC}{\HamConn} 
\DeclareMathOperator{\SHC}{\mathsf{StrHamCon}}

\newcommand{\DegF}[2]{\mathsf{\Deg}_{#1,\,#2}}
\newcommand{\DegSeq}{{\mathsf{DegSeq}}}

\newcommand{\PG}[3]{{P^{(#3)}}(#1,#2)}

\DeclareMathOperator{\val}{val}
\DeclareMathOperator{\con}{con}
\DeclareMathOperator{\anc}{anc}

\newcommand{\Eval}{\cE_{\val}}
\newcommand{\Econ}{\cE_{\con}}
\newcommand{\Eanc}{\cE_{\anc}}

\theoremstyle{definition}
\newtheorem*{connectivity}{Connectivity}
\newtheorem*{space-prop}{Space}
\newtheorem*{aperiodicity}{Aperiodicity}

\date{\today}

\title[Robust Hamiltonicity]{Robust Hamiltonicity}

\author[F.~Joos]{Felix Joos}
\author[R.~Lang]{Richard Lang}
\author[N.~Sanhueza-Matamala]{Nicolás Sanhueza-Matamala}

\thanks{The research leading to these results was supported by the Deutsche Forschungsgemeinschaft (DFG, German Research Foundation) -- 428212407 (F. Joos), H2020-MSCA -- 101018431 (R. Lang) and ANID-Chile through the FONDECYT Iniciación Nº11220269 grant (N. Sanhueza-Matamala).}

\address[F.~Joos]{Institut f\"ur Informatik, Universit\"at Heidelberg, Germany}
\email{joos@informatik.uni-heidelberg.de}

\address[R.~Lang]{Departament de Matemàtiques, Universitat Politècnica de Catalunya, Barcelona, Spain}
\email{richard.lang@upc.edu}

\address[N.~Sanhueza-Matamala]{Departamento de Ingeniería Matemática, Facultad de Ciencias Físicas y Matemáticas, Universidad de Concepción, Chile}
\email{nicolas@sanhueza.net}

\begin{document}

\begin{abstract}
	We study conditions under which a given hypergraph is randomly robust Hamiltonian, which means that a random sparsification of the host graph contains a Hamilton cycle with high probability.
	Our main contribution provides nearly optimal results whenever the host graph is Hamilton connected in a locally robust sense, which translates to a typical induced subgraph of constant order containing Hamilton paths between any pair of suitable ends.

	The proofs are based on the recent breakthrough on Talagrand's conjecture, which reduces the problem to specifying a distribution on the desired guest structure in the (deterministic) host structure.
	We find such a distribution via a new argument that reduces the problem to the case of perfect matchings in a higher uniformity.

	As applications, we obtain asymptotically optimal results for perfect tilings in graphs and hypergraphs both in the minimum degree and uniformly dense setting.
	We also prove random robustness for powers of cycles under asymptotically optimal minimum degrees and degree sequences.
	We solve the problem for loose and tight Hamilton cycles in hypergraphs under a range of asymptotic minimum degree conditions.
	This includes in particular $k$-uniform tight Hamilton cycles under minimum $d$-degree conditions for $1\leq k-d \leq 3$.
	In all cases, our bounds on the sparseness are essentially best-possible.
\end{abstract}

\maketitle

\thispagestyle{empty}
\vspace{-0.8cm}

\section{Introduction}

An old question in combinatorics is under which conditions a given graph $G$ contains a particular (vertex) spanning substructure.
For instance, Dirac's classic theorem~\cite{Dir52} determines the optimal requirements on the minimum degree for a Hamilton cycle: when $G$ has $n \geq 3$ vertices it suffices that $\delta(G) \geq n/2$, and this is best possible.
Over the past decades the study of minimum degree conditions and other simple conditions that force vertex spanning substructures has developed into its own branch of combinatorics.
While our understanding of these phenomena has greatly expanded, many open questions remain.

A related field in combinatorics concerns the study spanning substructures in random graphs.
For instance, Pósa~\cite{Pos76} and Korshunov~\cite{Kor76} showed, independently, that the binomial random graph $\oG(n,p)$ typically contains a Hamilton cycle  provided that $p$ is somewhat greater than $\log n /n$.

It is natural to ask whether these two branches can be combined.
The study of such \emph{robustness problems} was initiated by Krivelevich, Lee and Sudakov~\cite{KLS14}.
For an $n$-vertex graph $G$ with $\delta(G) \geq n/2$ and $p = \Omega(\log n / n)$, it was shown that the random sparsification $G_p$, obtained by keeping each edge of $G$ independently with probability $p$, typically contains a Hamilton cycle.
We investigate similar robustness problems for Hamiltonicity in hypergraphs.

Our main result (\cref{thm:main}) offers a framework for embedding a range of guest structures into host structures under asymptotically optimal conditions.

\smallskip

\textbf{Host structure.} Our results apply to classes of host (hyper)graphs that are hereditary in the sense of being approximately closed under taking random induced subgraphs of constant order.
This goes beyond the usually considered minimum degree conditions and includes for instance also classes of host graphs defined via degree (sequence) conditions and quasirandomness properties.

\smallskip

\textbf{Guest structure.} Our results embed spanning guest (hyper)graphs that are in a general sense cyclical.
Basically, we can embed anything that can be expressed by `chaining up' isomorphic copies of a fixed graph.
This includes perfect tilings, powers of cycles, hypergraph $\ell$-cycles and powers of hypergraph $\ell$-cycles.

\smallskip

\textbf{Connectivity.} A central challenge of Dirac-type problems for hypergraphs concerns connectivity.
For instance, finding tight Hamilton cycles beyond codegree conditions becomes much harder, since the host graph may no longer be connected~\cite{RRR19}.
To address this, we adopt a recently introduced setup for Hamiltonicity in dense graphs~\cite{LS24a}.
This allows us to find spanning structures in host graphs that are not entirely connected.
In particular, we solve the robustness problem for $3$-uniform tight Hamilton cycles under vertex degree conditions, which serves as a benchmark for overcoming connectivity issues.

\smallskip

\textbf{Proof techniques.} The proofs are based on the recent breakthrough on Talagrand's conjecture~\cite{FKN+21}, which reduces the problem to specifying a `spread' distribution on the desired guest structure in the (deterministic) host structure.
We achieve this via a new argument that reduces the problem to the case of perfect matchings in a higher uniformity.
Instead, we follow a recently introduced approach to study perfect tilings~\cite{Lan23} and Hamiltonicity~\cite{LS24a} in dense hypergraphs, which focuses on local substructures.

\subsection*{Overview}
The paper is organised as follows.
In the next section, we present a series of applications of our framework.
In \cref{sec:framework}, we motivate and state our main result.
The proof of the applications can be found in \cref{sec:applications-proofs}.
The proof of our main result spans \cref{sec:proof-main-result,sec:typical-to-correct-distribution-proof,sec:correct-to-spread-proof,section:strongspread-to-threshold}.
We conclude with a discussion and a few open problems in \cref{sec:conclusion}.

\section{Applications}
\label{sec:applications}

We present a series of applications that illustrate the use of our main result (\cref{thm:main});
in particular, we essentially obtain all known results where the guest graph has a cyclic structure and the host graph robustly contains this guest structure.
To illustrate the strength of our theorem we also mention results that recently have been obtained by Kelly, Müyesser and Pokrovskiy~\cite{KMP23} (\cref{thm:perfect-tilings-minimum-degree}, \cref{thm:power-hamiton-cycles-minimum-degree} and \cref{thm:tight-hamilton-cycles-codegree}).
We remark that their work on (powers of) Hamilton cycles only applies to situations where the host (hyper)graph consists of a single connectivity component.
This is however only given in certain specific situations.
In contrast, our work takes into account connectivity and thus leads to many further new results (\cref{thm:power-hamiton-cycles-degree-sequence,thm:perfect-tilings-uniformly-dense,thm:tight-hamilton-cycles-k-2-degree,thm:tight-hamilton-cycles-k-3-degree,thm:tight-hamilton-cycles-d-degree-framework}), including the answer to a question formulated by Kelly, Müyesser and Pokrovskiy (\cref{thm:tight-hamilton-cycles-k-2-degree}) that can be easily deduced from \cref{thm:main}.

In many scenarios it is also in question whether the optimal probability $p$ (as a function of $n$) contains polylogarithmic factors
and often there are different approaches necessary depending on the answer.
Another novelty of our result is that it covers both cases and which case actually applies can be easily read off from structural parameters of the guest graph.

\subsection*{Notation}

We write $[n] = \{1,\dots,n\}$.
A $k$-uniform hypergraph $G$ (\emph{$k$-graph} for short) consists of a set of vertices $V(G)$ and a set of edges $E(G)$, where each edge is a set of $k$ vertices (\emph{$k$-set} for short). Throughout, we assume that $k\in \bN\setminus \{1\}$.
We write $v(G)=|V(G)|$ for the \emph{order} of $G$ and $e(G) = |E(G)|$ for its \emph{size}.
For two hypergraphs $G,H$, we write $G\subseteq H$ if $V(G)\subseteq V(H)$ and $E(G)\subseteq E(H)$.
For $1 \leq d \leq k-1$, the {minimum $d$-degree}, written $\delta_d(G)$, is the maximum ${m} \in \NATS$ such that every set of $d$ vertices is contained in at least ${m}$ edges of~$G$.
For $d=k-1$, this is also called the \emph{minimum codegree}.

\newcommand{\ER}[3]{\oG^{(#3)}_{#1,#2}}
\newcommand{\Gnp}[1]{\ER{n}{p}{#1}}

\newcommand{\RS}[3]{{#1}^{#3}_{#2}}
\newcommand{\Gs}{G_p}

We denote by $\Gnp{k}$ the \emph{binomial random $k$-graph} on vertex set $[n]$, that is, each potential edge is independently present with probability $p$.
For a $k$-graph $G$ and $p\in [0,1]$, let $\Gs$ be the \emph{randomly sparsified graph} which arises from $G$ by retaining each edge independently with probability $p$.
Thus, $G_p = \Gnp{k}$ when $G$ is the complete $k$-graph on vertex set $[n]$.
We say that $G' \sim \Gnp{k}$ satisfies a property $\cP$  \emph{with high probability} if $G'$ satisfies $\cP$ with probability tending to $1$ as $n$ goes to infinity.
The same formulation is used for random sparsifications $G' \sim G_p$, where we implicitly assume that~$G'$ can be specified for arbitrarily large values of $n$.

\subsection*{Perfect tilings}
\label{sec:applications-tilings}

Our first set of outcomes concerns tilings, which present a generalisation of matchings.
Given $k$-graphs $F$ and $G$, a \emph{perfect $F$-tiling} of $G$ is a set of copies of $F$ in $G$ (called \emph{tiles}) such that every vertex of $G$ is on exactly one copy. (Whenever we consider perfect $F$-tilings in a graph $G$, we tacitly assume that $v(F)$ divides $v(G)$.)
We are interested in (asymptotic) Dirac-type results for perfect tilings.
It is convenient to express this in terms of `thresholds'.
For an $m$-vertex $k$-graph $F$, we define the \emph{minimum $d$-degree threshold for perfect $F$-tilings}, denoted by $\th_d^{}(\til_F)$, as the infimum of $\delta \in [0,1]$ such that for all $\mu>0$ and $n$ large enough,
every $n$-vertex $k$-graph $G$ with $\delta_d(G) \geq (\delta +\mu) \binom{n-d}{k-d}$ admits a perfect $F$-tiling.

The thresholds for $F$-tilings in $2$-graphs were determined by Kühn and Osthus~\cite{KO09} after preliminary work of many others.
(A more detailed trajectory of the problem can be found in the survey of Simonovits and Szemerédi~\cite{SS19}.)
For hypergraphs, fewer results are known (see~\cite{Lan23} for a recent survey).
Our first application lifts these results to the random robust setting, even in those cases where the thresholds have not yet been determined.

For random graphs, Johansson, Kahn and Vu~\cite{JKV08} 
investigated $F$-tilings in $\ER{n}{p}{k}$.
Define $d_1(F)= e(F)/(v(F)-1)$ as the \emph{$1$-density} of $F$, and say that a graph is \emph{strictly $1$-balanced} if $d_1(F')<d_1(F)$ for all proper subgraphs $F'$ of $F$.
Most hypergraphs of interest are strictly $1$-balanced, such as cycles and complete graphs, and it is natural to restrict first the attention to these graphs as here the number of copies that contain a fixed vertex is well-behaved for all~$p$.
Johansson, Kahn and Vu determined when $\ER{n}{p}{k}$ contains a perfect $F$-tiling for strictly $1$-balanced $k$-graphs $F$.

In the setting of robustness problems, it was recently shown~\cite{ABC+22,KKK+22,PSS+22} that for graphs $G$ on~$n$ vertices with $\delta_1(G) \geq (1-1/r)n$,
the random sparsification $\Gs$ typically contains a perfect $K_r$-tiling, where $p$ can be as low as we can guarantee that each vertex is contained in a copy of $K_r$.
We extend these results for asymptotically optimal degree conditions by replacing $K_r$ with a strictly $1$-balanced $F$ and also from graphs to hypergraphs.
This result was also obtained independently by Kelly, Müyesser and Pokrovskiy~\cite{KMP23}.

\begin{theorem}[Perfect tilings -- minimum degree]\label{thm:perfect-tilings-minimum-degree}
	For all $1 \leq d \leq k-1$, $\mu>0$ and $1$-balanced $k$-graphs $F$, there is $C>0$ with the following property.
	Let $G$ be an $n$-vertex $k$-graph with $\delta_d(G)\geq (\th_d^{}(\til_F)+\mu)\binom{n-d}{k-d}$.
	Suppose $p \geq C(\log n)^{1/e(F)}n^{-1/d_1(F)}$ and $n$ is divisible by $v(F)$.
	Then $\Gs$ contains a perfect $F$-tiling with high probability.
\end{theorem}

The bound on $p$ is best-possible up to the leading constant.
This is because it coincides with the threshold where we can ensure that each vertex belongs to a copy of $F$, which is clearly a necessary condition to contain a perfect $F$-tiling.
Hence, $k$-graphs above the corresponding minimum degree threshold contain a perfect $F$-tiling in a similarly plentiful way as a perfect clique-tiling.

\subsection*{Beyond degree conditions}

Our results are not restricted to host graphs from the setting of Dirac's theorem.
In fact, we shall prove a generalisation of \cref{thm:perfect-tilings-minimum-degree} that covers all hypergraph families which are approximately closed under taking typical induced subgraphs of constant order.
This includes for instance the setting of degree sequences and quasirandom graphs.
We illustrate this at the example of perfect tilings in uniformly dense hypergraphs.

\newcommand{\DenF}[2]{\mathsf{\den}_{#1,\, #2}}

For $\eps, d >0$, we say that an $n$-vertex $k$-graph $G$ is \emph{uniformly $(\eps,d)$-dense} if for all (not necessarily disjoint) $X_1,\dots, X_k \subset V(G)$, we have
\begin{align*}
	e_G(X_1,\dots,X_k) \geq d |X_1|\cdots |X_k| - \eps n
\end{align*}
where $e_G(X_1,\dots,X_k)$ is the number of tuples $\{v_1,\dots,v_k\} \in E(H)$ with $v_i \in X_i$ for each $ 1\leq i \leq k$.
Let $\DenF{\eps}{d}$ be the set of all $(\eps,d)$-uniformly dense $k$-graphs.

Lenz and Mubayi~\cite{LM16} asked for which tiles $F$ every large enough quasirandom hypergraph is guaranteed to contain a perfect $F$-tiling.
Since uniform density does not prevent isolated vertices, one has to add a mild degree condition for any chances of success.
Formally, let $\cF({\til})$ contain,
for every $k\geq 2$, the $k$-graphs $F$ such that for every $d,\mu >0$, there is $\eps >0$  such that for every large enough $n$ divisible by $v(F)$,
every $n$-vertex $k$-graph $G \in {\DenF{\eps}{d}^{}}$ with $\delta_1{(G)} \geq \mu
	\binom{n-1}{k-1}$ has a perfect $F$-tiling.

Lenz and Mubayi~\cite{LM16} initiated the study of the class $\cF({\til})$ and showed that it contains all linear hypergraphs.
More recently, Ding, Han, Sun, Wang and Zhou~\cite{DHS+22} provided a characterisation of all $k$-partite $k$-graphs and all $3$-graphs in $\cF({\til})$.

We obtain the following robustness result in this setting.

\begin{theorem}[Perfect tilings -- uniform density]\label{thm:perfect-tilings-uniformly-dense}
	For all $k\geq 2$, $d,\mu > 0$ and $1$-balanced $k$-graphs~$F$ in $\cF({\til})$, there are $C,\, \eps>0$ with the following property.
	Let $G \in {\DenF{\eps}{d}^{}}$ be an $n$-vertex $k$-graph with $\delta_1(G) \geq \mu
		n^{k-1}$.
	Suppose $p \geq C(\log n)^{1/e(F)}n^{-1/d_1(F)}$ and $n$ is divisible by $v(F)$.
	Then $\Gs$ contains a perfect $F$-tiling with high probability.
\end{theorem}

\subsection*{Powers of cycles}
\label{sec:applications-power}
A graph $C$ is \emph{the $\ell$th power of a cycle} if there is a cyclical ordering of the vertices of $C$ such that every set of $\ell+1$ consecutive vertices forms a clique.
Moreover, $C \subset G$ is \emph{Hamilton} in a host graph $G$ if it spans all its vertices.
We define the \emph{minimum degree threshold for the $\ell$th power of a Hamilton cycle}, denoted (with hindsight) by $\th_1^{}(\powham^\ell)$,
as the infimum over all $\delta \in [0,1]$ such that for all $\mu>0$ and $n$ large enough,
every $n$-vertex $2$-graph $G$ with $\delta_1(G) \geq (\delta +\mu) n$ contains the $\ell$th power of a Hamilton cycle.
It was shown by Koml\'{o}s, S\'{a}rk\"{o}zy and Szemer{\'e}di~\cite{KSS98} that $\th_1^{}(\powham^\ell) = 1-1/(\ell+1)$ for all $\ell\geq 1$, which generalises Dirac's~theorem.

As mentioned above, $\Gnp{2}$ contains a Hamilton cycle with high probability provided that $p$ is somewhat greater than $\log n/n$.
For the $\ell$th power of a Hamilton cycle, similar results were derived by Kühn and Osthus~\cite{KO12} for $\ell\geq 3$ from a result of Riordan~\cite{Rio00}.
The remaining case $\ell=2$ was recently resolved by Kahn, Narayanan and Park~\cite{KNP21}. 
In combination, these results show that $\Gnp{2}$ contains the $\ell$th power of a Hamilton cycle for all $\ell\geq2$ if $p=\Omega(n^{-1/\ell})$.
For more details, we refer the reader to the recent survey of Frieze~\cite{Fri19}.

Here we give a common generalisation of these results.
The following result was also obtained independently by Kelly, Müyesser and Pokrovskiy~\cite{KMP23}.

\begin{theorem}[Powers of Hamilton cycles -- minimum degree]\label{thm:power-hamiton-cycles-minimum-degree}
	Let $k,\ell \geq 2$ and $\mu>0$.
	Suppose that $p = \omega(n^{-1/\ell})$.
	Let $G$ be an $n$-vertex $2$-graph with $\delta_1(G) \geq (\th_1^{}(\powham^\ell)+\mu)n$.
	Then $\Gs$ contains the $\ell$th power of a Hamilton cycle with high probability.
\end{theorem}

Pósa~\cite{Pos63} extended Dirac's theorem by showing that a graph on $n\geq3$ vertices contains a Hamilton cycle provided that its degree sequence $d_1 \leq \dots \leq d_n$ satisfies $d_i > i + 1$ for all $i \leq n/2$.
This was recently generalised to powers of Hamilton cycles~\cite{LS23,ST17}.
Our next theorem gives a robustness version of this result.

\begin{theorem}[Powers of Hamilton cycles -- degree sequence]\label{thm:power-hamiton-cycles-degree-sequence}
	Let $\ell \geq 2$, $\mu>0$,
	and suppose that $p = \omega(n^{-1/\ell})$.
	Let $G$ be a graph with degree sequence $d_1 \leq \dots \leq d_n$.
	Suppose that $d_i > (\ell-1)n/(\ell+1) + i + \mu n$ for every $i \leq n/(\ell+1)$.
	Then $\Gs$ contains the $\ell$th power of a Hamilton cycle with high probability.
\end{theorem}

\subsection*{Tight cycles}
\label{sec:applications-tight}

We can study the $\ell$th power of a Hamilton cycle by considering an $(\ell+1)$-graph in which each edge plays the role of an $(\ell+1)$-clique.
This leads to the following notion of cycles.
A \emph{tight Hamilton cycle} in a $k$-graph $G$ is an ordering of the vertices of $G$ such that every $k$ consecutive vertices form an edge.
We define the \emph{minimum degree threshold for $k$-uniform tight Hamilton cycles}, denoted by $\th_d^{}(\ham_{k})$, as the infimum over all $\delta \in [0,1]$ such that for all $\mu>0$ and large enough $n$, every $n$-vertex $k$-graph $G$ with $\delta_d(G) \geq (\delta +\mu) \binom{n-d}{k-d}$ contains a tight Hamilton cycle.
It was shown by Rödl, Rucinski and Szemerédi~\cite{RRS09a} that $\th_{k-1}^{}(\ham_{k})=1/2$ in (another) extension of Dirac's theorem.
Reiher, Rödl, Ruciński, Schacht and Szemerédi~\cite{RRR19} proved that $\th_1(\ham_3)=5/9$ after preliminary work of many others, which resolves the case of $d=k-2$ when $k=3$.
Subsequently, this was generalised to $k=4$~\cite{PRR+20} and finally, Polcyn, Reiher, Ruciński and Schülke~\cite{PRRS21} and, independently, Lang and Sanhueza-Matamala~\cite{LS22} established $\th_{k-2}(\ham_k)=5/9$ for all $k \geq 3$.
Recently, it was proved by Lang, Schacht and Volec~\cite{LSV24} that $\th_{k-3}(\ham_k)=5/8$ for all $k \geq 4$.

The threshold for tight Hamilton cycles in binomial random $k$-graphs was determined by Dudek and Frieze~\cite{DF13}, who showed that if $p$ is slightly above $e/n$ then $\Gnp{k}$ contains a tight Hamilton cycle.

We prove the following approximately tight result, which was obtained independently by Kelly, Müyesser and Pokrovskiy~\cite{KMP23}.

\begin{theorem}[Tight Hamiltonicity I]\label{thm:tight-hamilton-cycles-codegree}
	Let $k\geq 2$ and $\mu>0$, and suppose that $p = \omega(1/n)$.
	Let $G$ be an $n$-vertex $k$-graph with $\delta_{k-1}(G) \geq (1/2+\mu)n$.
	Then $\Gs$ contains a tight Hamilton cycle with high probability.
\end{theorem}

Our techniques also allow us to work in host hypergraphs beyond codegree conditions, which was a barrier in previous works.
In particular, we obtain the following result about $k$-graphs with large $(k-2)$-degree, which was posed as an open problem by Kelly, Müyesser and Pokrovskiy \cite[Section 8.2]{KMP23}.

\begin{theorem}[Tight Hamiltonicity II]\label{thm:tight-hamilton-cycles-k-2-degree}
	Let $k\geq 3$ and $\mu>0$ and suppose that $p = \omega(1/n)$.
	Let $G$ be an $n$-vertex $k$-graph with $\delta_{k-2}(G) \geq (5/9+\mu)\binom{n-2}{k-2}$.
	Then $\Gs$ contains a tight Hamilton cycle with high probability.
\end{theorem}

Our techniques in fact allow us to extend even the most recent advances and obtain a random robust version for $k$-graphs under $(k-3)$-degree conditions.

\begin{theorem}[Tight Hamiltonicity III]\label{thm:tight-hamilton-cycles-k-3-degree}
	Let $k\geq 4$ and $\mu>0$ and suppose that $p = \omega(1/n)$.
	Let~$G$ be an $n$-vertex $k$-graph with $\delta_{k-3}(G) \geq (5/8+\mu)\binom{n-3}{k-3}$.
	Then $\Gs$ contains a tight Hamilton cycle with high probability.
\end{theorem}

\subsection*{Hamilton frameworks}
There is a substantial difference in the nature of \cref{thm:tight-hamilton-cycles-codegree}  on the one hand and  \cref{thm:tight-hamilton-cycles-k-2-degree,thm:tight-hamilton-cycles-k-3-degree} on the other hand, because unlike the former the latter require a careful analysis of the connectivity structure of hypergraphs.
{Our main result, allows us to capture this difficulty in a concise way.}
(We return to this discussion in \cref{sec:framework}.)
To formalise this, we introduce the notion of \emph{Hamilton frameworks} following the exposition of  Lang and Sanhueza-Matamala~\cite{LS24a}.

We begin with a discussion of three structural features that constitute Hamiltonicity: {connectivity}, {space} and {aperiodicity}.
Let $G$ be a $k$-graph on $n$ vertices.
 
\begin{connectivity}
	We denote by $L(G)$ the \emph{line graph} of $G$, which is the $2$-graph on vertex set $E(G)$ with an edge $ef$ whenever $|e \cap f|=k-1$.
	A subgraph of $G$ is \emph{tightly connected} if it has no isolated vertices and its edges induce a connected subgraph in $L(G)$.
	Moreover, we refer to edge-maximal tightly connected subgraphs as \emph{tight components}.
\end{connectivity}

\begin{space-prop}
	A \emph{fractional matching} is a function $\omega\colon E(G) \to [0,1]$ such that $\sum_{e\colon v \in e} \omega (e) \leq 1$ for every vertex $v \in V(G)$.
	The \emph{size} of a fractional matching $\omega$ is $\sum_{e \in E(G)} \omega (e)$.
	We say that $\omega$ is \emph{perfect} if its size is $n/k$.
\end{space-prop} 

\begin{aperiodicity}
	A \emph{homomorphism} from a $k$-graph $C$ to $G$ is a function $\phi \colon V(C) \to V(G)$ that maps edges to edges.
	We say that $W \subset G$ is a \emph{closed tight walk} if $W$ is the image of homomorphism of a tight cycle $C$.
	The \emph{order} of $W$ is the order of $C$.
	We call $G$ \emph{aperiodic} if it contains a closed walk whose order is congruent to $1$ modulo $k$.
\end{aperiodicity}

It is easy to see that any tight cycle is connected, contains a perfect fractional matching and (if its order is coprime to $k$) is aperiodic.
In other words, these three properties satisfied by any family~$\cP$ of $k$-graphs whose members are Hamiltonian after deleting up to $k-1$ vertices.
In addition to this, constructions show that such a family $\cP$ must also satisfy a fourth property, which ensures that these features are \emph{consistent} between reasonable similar member of $\cP$.
This motivates the following definition.

\begin{definition}[Hamilton framework]\label{def:hamilton-framework}
	A family $\cP$ of $s$-vertex $k$-graphs has a \emph{Hamilton framework}~$F$ if for every $G \in \cP$, there is an $s$-vertex subgraph $F(G) \subset G$ such that
	\begin{enumerate}[(F1)]
		\item \label{itm:hf-connected} $F(G)$ is a tight component, \hfill(connectivity)
		\item \label{itm:hf-matching} $F(G)$ has a perfect fractional matching, \hfill(space)
		\item \label{itm:hf-odd} $F(G)$ contains a closed walk of order $1 \bmod k$, and \hfill(aperiodicity)
		\item \label{itm:hf-intersecting} $F(G) \cup F(G')$ is tightly connected whenever $G,G' \in \cG$ are obtained by deleting a distinct vertex from the same $(s+1)$-vertex $k$-graph. \hfill(consistency)
	\end{enumerate}
\end{definition}

Thus, a Hamilton framework is a family of graphs which satisfy the four aforementioned necessary properties for Hamiltonicity.
It is not true that every member of a Hamilton framework is Hamiltonian, but, as we will see, a natural hardening of these properties does indeed imply Hamiltonicity.
The following definition connects Hamilton frameworks to degree conditions.
For a $k$-graph $G$ and $t\geq k$, the \emph{clique-graph} $K_t(G)$ has vertex set $V(G)$ and a $t$-uniform edge $X$ whenever~$G[X]$ induces a clique.
Note that a $(t-k+1)$st power of a tight Hamilton cycle in $G$ corresponds to a tight Hamilton cycle in $K_t(G)$.

\begin{definition}\label{def:ham-fw-threshold}
	Let $\delta_d^{\hf}(k,t)$ be the infimum $\delta \in [0,1]$ such for every $\eps$, there is $n_0$ such that the family of $t$-graphs $K_t(G)$ on $n \geq n_0$ vertices with $\delta_d(G) \geq (\delta + \eps ) \binom{n-d}{k-d}$ admits a Hamilton framework.
\end{definition}

It was proved by Lang and Sanhueza-Matamala~\cite{LS24a} that $\delta_{k-1}^{\hf}(k) = 1/2$ $\delta_{k-2}^{\hf}(k) = 5/9$ and $\delta_{k-3}^{\hf}(k) = 5/8$, as well as $\delta_{1}^{\hf}(2,t) = \th_1^{}(\powham^{t-1})$.
Hence \cref{thm:power-hamiton-cycles-minimum-degree,thm:tight-hamilton-cycles-codegree,thm:tight-hamilton-cycles-k-2-degree,thm:tight-hamilton-cycles-k-3-degree} are corollaries of the following much more general result.

\begin{theorem}[Tight Hamiltonicity IV]\label{thm:tight-hamilton-cycles-d-degree-framework}
		Let $1 \leq d < k$, $\mu>0$ and $\delta = \delta_d^{\hf}(k,t)$, and suppose that $p = \omega(1/n)$.
		Let $G$ be an $n$-vertex $k$-graph with $\delta_{d}(G) \geq (\delta+\mu) \binom{n-d}{k-d}$.
		Then $\Gs$ contains a tight Hamilton cycle with high probability.
\end{theorem}

We remark that \cref{thm:tight-hamilton-cycles-d-degree-framework} delivers optimal results for robust Hamiltonicity whenever $\delta_d^{\hf}(k,k) = \th_d^{}(\ham_{k})$, which is conjectured to hold for all $1\leq d < k$~\cite[Conjecture 2.7]{LS24a}.
In fact, we shall derive \cref{thm:tight-hamilton-cycles-d-degree-framework} from an even more general result (\cref{thm:framework-connectedness}), which goes beyond the setting of minimum degree conditions and also captures powers of tight cycles.
The defer the details to \cref{sec:frameworks}.

\subsection*{Loose cycles}
\label{sec:applications-loose}

We can relax the definition of tight Hamiltonicity by allowing $k$-uniform edges to interlock in fewer than $k-1$ vertices.
More precisely, for $0 \leq \ell \leq k-1$ an \emph{$\ell$-cycle}
$C$ in a $k$-graph $G$ has its vertices cyclically ordered such that each edge consists of $k$ consecutive vertices and consecutive edges intersect in exactly $\ell$ vertices (an $1$-cycle is also called a \emph{loose} cycle).
We say that $C$ is \email{Hamilton} if it covers all vertices of $G$.
(Note that the case $\ell=0$ corresponds to a perfect matching.)
Similarly as for tilings, we tacitly assume when considering Hamilton $\ell$-cycles in a $k$-graph on $n$ vertices that $k-\ell$ divides $n$.
We define the \emph{minimum degree threshold for $k$-uniform Hamilton $\ell$-cycles}, denoted by $\th_d^{}(\ham_{k,\ell})$,
as the infimum over all $\delta \in [0,1]$ such that for all $\mu>0$ and large enough $n$,
every $n$-vertex $k$-graph $G$ with $\delta_d(G) \geq (\delta +\mu) \binom{n-d}{k-d}$ contains a Hamilton $\ell$-cycle.
It was shown by Kühn, Mycroft and Osthus~\cite{KMO10} that $\th_{k-1}^{}(\ham_{k,\ell}) = 1/(\lceil k/(k-\ell)\rceil)(k-\ell))$ whenever $k-\ell$ does not divide $k$ and $\th_{k-1}^{}(\ham_{k,\ell})=1/2$ otherwise.
Most known bounds on $\th_d^{}(\ham_{k,\ell})$ for $d \leq k-2$ assume that $\ell \leq k/2$ meaning cycles in which no vertex is contained in three or more edges.
Specifically, the threshold for $\ell = k/2$ and $k/2 \leq d \leq k-2$ is determined in~\cite{BMSSS17,HHZ20} and for some family of values $1 \leq \ell < k/2$ in~\cite{GHS+21}.

The problem of finding Hamilton $\ell$-cycles in $\Gnp{k}$ was resolved in~\cite{DF13,NS20}.
In particular, it is known that for $\ell \geq 2$ it suffices to take $p = \Omega(n^{-k+\ell})$, while for $\ell=1$ (loose cycles) we require  $p= \Omega (n^{-k+1}\log n)$.
We confirm the same for the robustness setting and for loose cycles, we can in fact cover the full degree spectrum.

\begin{theorem}[Hamilton $\ell$-cycles]\label{thm:loose-hamilton-cycles-codegree}
	Let $2 \leq \ell \leq k-1$ and $\mu>0$ and suppose that $p = \omega(n^{-k+\ell})$.
	Let {$n$ be divisible by $k - \ell$, and let} $G$ be an $n$-vertex $k$-graph with $\delta_{k-1}(G) \geq (\th_{k-1}^{}(\ham_{k,\ell})+\mu)n$.
	Then $\Gs$ contains a Hamilton $\ell$-cycle with high probability.
\end{theorem}

\begin{theorem}[Loose Hamiltonicity]\label{thm:loose-hamilton-cycles-l=1}
	For every $1\leq d \leq k-1$ and $\mu>0$, there is $C >0$ with the following property.
	Suppose that $p \geq C n^{-k+1} \log n$.
	Let {$n$ be divisible by $k - 1$, and let} $G$ be an $n$-vertex $k$-graph with $\delta_{d}(G) \geq (\th_{d}^{}(\ham_{k,1})+\mu) \binom{n-d}{k-d}$.
	Then $\Gs$ contains a loose Hamilton cycle with high probability.
\end{theorem}

\section{A framework for robust Hamiltonicity}
\label{sec:framework}

In this section, we present our main result.
To hide some of the technical details, we begin with a simplified version for perfect matchings, which play a particularly important rule.

\subsection*{Perfect matchings}

Informally, our results assume that the (deterministic) host graphs have perfect matchings in a hereditary sense,
which means that typical induced subgraphs of constant order contain a perfect matching.
This can be formalised in terms of the recently introduced notion of property graphs~\cite{Lan23}.

\begin{definition}[Property graph]\label{def:property-graph}
	For a $k$-graph $G$ and a family of $k$-graphs $\sP$, the \emph{property graph} $\PG{G}{\sP}{s}$ is the $s$-graph with vertex set $V(G)$ and $S \subset V(G)$ is an edge if $|S|=s$ and $G[S]$ \emph{satisfies}~$\sP$, that is $G[S] \in \sP$.
\end{definition}

We denote by $\mat_k$ the family of all $k$-graphs that have a perfect matching.
Given this, we can already state our main result for perfect matchings.

\begin{theorem}[Robust perfect matchings]
	\label{thm:main-simple}
	For all $k, s_0$ with $s_0 \geq 5k$, there are $C>0$ and $s_1 \geq s_0$ such that the following holds.
	Suppose $G$ is an $n$-vertex $k$-graph and $\delta_1 \big(P\big) \geq  \left(1-s^{-2}\right)  \tbinom{n-1}{s-1}$ for  each $s_0 \leq s \leq s_1$ divisible by $k$,
	where $P = \PG{G}{\mat_k}{s}$.
	Then~$\Gs$ contains a perfect matching with high probability
	whenever $p \geq C n^{-k+1} \log n$ and $k$ divides $n$.
\end{theorem}

Suppose $F$ is an $m$-vertex $k$-graph.
Note that an $F$-tiling in a $k$-graph $G$ can be viewed as a matching in an auxiliary $m$-graph on the same vertex set.
From this angle, it is not hard to see that \cref{thm:main-simple} together with the work of Riordan~\cite{Rio22} (see \cref{thm:rio}) already implies \cref{thm:perfect-tilings-minimum-degree,thm:perfect-tilings-uniformly-dense}.
The details can be found in \cref{sec:applications-proofs}.

\subsection*{Motivation}
Now let us turn to connected structures such as Hamilton cycles.
Our framework is stated in terms of non-uniform and directed hypergraphs.
This language has been adopted from a recent comprehensive study of Hamiltonicity in dense graphs~\cite{LS24a}.
Before we come to the details, let us briefly explain why this is a suitable setup.

For given $\mu>0$ and $n$ sufficiently large, consider a $3$-graph $G$ on $n$ vertices, and recall from \cref{sec:applications-tight} that $G$ contains a tight Hamilton cycle if $\delta_2(G) \geq (1/2 + \mu) n$ or $\delta_1(G) \geq (5/9 + \mu) \binom{n}{2}$.
It follows from the arguments of Rödl, Ruci\'nski and Szemerédi~\cite{RRS09a} that in the first case, one can in fact connect any pair of (disjoint) ordered edges with a tight Hamilton path.
So when $\delta_2(G) \geq (1/2 + \mu) n$, then $G$ satisfies a strong form of `Hamilton connectedness'.
This is no longer true for $\delta_1(G) \geq (5/9 + \mu) \binom{n}{2}$, as under this condition we may encounter edges that are not connected by any tight path~\cite{RRR19}.
The strategy to find a tight Hamilton cycle in $G$ therefore consists in first identifying a subgraph $H \subset G$ together with a subset of ordered edges, which are suitably connected.
(In fact, it suffices to consider `connectable' pairs of vertices.)
One then shows that $H$ contains a tight Hamilton cycle using only ordered edges.\footnote{We remark that the technical and conceptual details of this argument have so far lead to quite involved proofs~\cite{LS22,PRR+20,PRRS21,RRR19}.
	In the our~\cite{LS24a}, a new approach to this problem is taken, which leads to a more efficient proof and a more general picture.}
Our setup is therefore directly formulated in terms a $3$-graph $H$ together with a set of `connectable' pairs of vertices.
Moreover, once we have internalised this perspective, it becomes natural to state our main result (\cref{thm:main}) in terms of a non-uniform directed hypergraph, whose edges consist of $k$-tuples (corresponding to the edges of $H$) and $j$-tuples (corresponding to the connectable tuples of $H$).

We now turn to the details.
A \emph{directed hypergraph} (\emph{digraph} for short) $G$ consists of a set of vertices $V(G)$ and a set of edges $E(G)$ where each \emph{edge} is a tuple of vertices without repetitions.
If~$E(G)$ contains only edges of \emph{uniformity} $k$ (that is, only $k$-tuples), we speak of a \emph{$k$-digraph}.
Similarly, if $E(G)$ contains only tuples of uniformity $k$ and $\ell$, we speak of a \emph{$(k,\ell)$-digraph}.
We let $G^{(k)}$ be the $k$-digraph on vertex set $V(G)$ whose edge set is restricted to the $k$-edges of $G$.
Concepts such as (induced) subgraphs extend to digraphs in the obvious way.
For $p \in [0,1]$, we denote by $\RS{G}{p}{(k)}$ the randomly sparsified $k$-digraph obtained by keeping every $k$-edge of $G$ with probability $p$ and all edges of uniformity other than $k$ with probability $1$.

We remark that \cref{def:property-graph} extends straightforward to digraph families $\sP$.
More precisely, $\PG{H}{\sP}{s}$, is the (undirected) $s$-graph on vertex set $V(G)$ with an edge $S \subset V(H)$ whenever the induced subgraph $G[S]$ \emph{satisfies}~$\sP$, that is $G[S] \in \sP$.

\subsection*{Links and chains}

Now we formalise the class of Hamilton cycles that can be embedded by our main result.
We view these structures as (closed) `chains' whose `links' are isomorphic (ordered) hypergraphs.
As we shall see, this allows us to capture all the situations presented in \cref{sec:applications} (and many more).

A \emph{$(k,\ell,r)$-link} $L$ is a directed $k$-graph on $r+\ell \geq k$ vertices whose vertices are ordered such that the last $\ell$ vertices induce an independent set.
(In the context of $(k,\ell,r)$-links, we always assume that $k \geq 2$, $r\geq 1$ and $\ell \geq 0$ with $r+\ell\geq k$.)
An \emph{$L$-chain} is obtained by connecting copies of $L$ such that consecutive links intersect in the first and last $\ell$ vertices, respectively.
More precisely, in an $L$-chain on vertex set $[rm+\ell]$ (equipped with the natural ordering) every vertex sequence $ir+1,\dots,ir+r+\ell$ for $0\leq i < m$ induces an order preserving copy of $L$.\footnote{By \emph{order preserving}, we mean that $j$th vertex of $L$ is mapped to the $j$th vertex of the sequence.}
A \emph{closed $L$-chain} on vertex set $[rm]$, where $rm\geq 2(r+\ell)$,
is a directed $k$-graph such that every vertex sequence $ir+1,\dots,ir+r+\ell$ for $0\leq i < m$ induces an order preserving copy of $L$ (index computation modulo $rm$).
An $L$-chain whose vertex ordering starts with $X$ and ends with $Y$
is a \emph{Hamilton $L$-chain} from~$X$ to $Y$ in some host graph $G$ if it contains all vertices of $G$.
Note that a closed $L$-chain $C$ on $n$ vertices consists of $n/r$ copies of $L$.
Moreover, $e(C) / n = e(L) /r$.
It is easy to see that the spanning structures from \cref{sec:applications} (and many others) can be expressed as $L$-chains for suitable parameters.

\begin{example}\label{exa:chains}
	We can model $\ell$-cycles in $k$-graphs, by taking $L$ to be the $(k,\ell,k-\ell)$-link on vertex set $\{1,\dots,k\}$ consisting of a single edge $(1,\dots,k)$.
	(Note that the case $\ell=0$ corresponds to matchings.)
	Similarly, the $\ell$th power of a $2$-uniform cycle can be expressed by taking $L$ to be the $(2,\ell,1)$-link on vertex set $\{1,\dots,\ell+1\}$ with an edge $(1,i)$ for every $2 \leq  i \leq \ell+1$.
\end{example}

Next, we define the property of `Hamilton connectedness', which we intend to use in conjunction with the property graph.

\begin{definition}[Hamilton connectedness]
	Let $L$ be a $(k, \ell, r)$-link.
	A $(k,\ell)$-digraph $G$ is \emph{Hamilton $L$-connected} if $G^{(\ell)}$ contains at least two disjoint $\ell$-tuples
	and for every two disjoint $\ell$-tuples $X,Y$ in~$G^{(\ell)}$, there is a  Hamilton $L$-chain from $X$ to $Y$ in $G$.
	Denote by $\HamConn_L$ the family of Hamilton $L$-connected digraphs, and note that these digraphs have order congruent to $\ell$ modulo $r$.
\end{definition}

Finally, we require the following property of $L$-chains, which plays a conceptually similar role to $1$-density in random graph embedding problems.
We say that an $n$-vertex $L$-chain $C$ is \emph{$(d, \lambda)$-balanced} if for each $I$ subgraph of $C$ on $v$ vertices and $e > 0$ edges,
we have
\begin{enumerate}[\upshape{(B\arabic*)}]
	\item \label{balanced:global} $e \leq dv$, and
	\item \label{balanced:small} if $v \leq \frac{n}{2v(L)}$, then also $e \leq dv - \lambda$.
\end{enumerate}

Most cases considered in the literature fulfil this property for suitable choices of $d$ and $\lambda$.
In fact, we also verify later that this is true for all the examples of $L$-chains mentioned above.
The details can be found in \cref{sec:applications-proofs}.

\subsection*{Main result}

Now we are ready to present our main result.
It states that
given a host $(k,\ell)$-digraph $G$ which is locally Hamilton $L$-connected (captured via the property graph),
then $\RS{G}{p}{(k)}$ contains a closed Hamilton $L$-chain with high probability, where $p$ is within a constant factor of the optimal value.

\begin{theorem}[Robustness of chains -- Main result]\label{thm:main}
	For each $k \geq  2$, $r \geq  1$, $\ell \geq 0$, $s_1 \geq  5(r+\ell) \geq  k$ and $q=2 \ell + 1$, there is $C>0$ such that the following holds.
	Let $d,\lambda>0$ and let $L$ be a $(k, \ell, r)$-link such that sufficiently large closed $L$-chains are $(d, \lambda)$-balanced.
	Let~$G$ be an $n$-vertex $(k,\ell)$-digraph where $n$ is divisible by $r$.
	Suppose that  for every $s_1 \leq s \leq 4 s_1$ with $s \equiv \ell \bmod r$, the $s$-graph $P = \PG{G}{\HamConn_L}{s}$ satisfies $\delta_{q}(P) \geq (1 - s^{-2}) \binom{n-q}{s - q}$.
	Suppose further that
	\begin{enumerate}[\upshape (i)]
		\item \label{item:main-general} $\lambda \geq d$ and $p \geq C n^{-1/d} \log n$
		      or
		\item \label{item:main-balanced} $\lambda>d$ and $p = \omega(n^{-1/d})$.
	\end{enumerate}
	Then $G' \sim \RS{G}{p}{(k)}$ contains a closed Hamilton $L$-chain with high probability.
\end{theorem}

In \cref{sec:applications-proofs} it is shown that \cref{thm:main-simple} can easily be derived from this result.

\subsection*{Outline of the proof}

The proof of \cref{thm:main} relies on the recent breakthrough of Frankston, Kahn, Narayanan and Park~\cite{FKN+21} on Talagrand's conjecture.
Roughly speaking, this reduces the problem of finding a closed Hamilton $L$-chain in  $G' \sim \RS{G}{p}{(k)}$ to finding a suitably `spread' distribution on closed Hamilton $L$-chains of $G$.

Our key contribution (\cref{lem:typical-to-correct-distribution}) therefore concerns the existence of such a `spread' distribution.
In the proof, we specify a random procedure that results in a closed Hamilton $L$-chain in $G'$ as follows.
We first cover around half the vertices with short $L$-chains.
Then we use the remaining uncovered vertices cyclically to connect the $L$-chains of the first step with other $L$-chains.
As it turns out, both of these steps can be implemented by solving a perfect matching problem in suitable variants of the property $s$-graph $P$  of $\HamConn_L$ in $G$.
The heart of the argument then consists in showing that the distribution arising from this procedure is `spread' provided that the property graph $P$ has itself has a `spread' distribution for perfect matchings (\cref{lem:spread-property-matching-to-correct-distribution}).
In other words, we reduce the problem to the setting of minimum degree conditions for perfect matchings, which explains the degree condition on $P$ in \cref{thm:main}.
Suitably `spread' distribution on perfect matchings under minimum degree conditions are known to exist due to the work of Kang, Kelly, K{\"u}hn, Osthus and Pfenninger~\cite{KKK+22}, Pham, Sah, Sawhney, Simkin~\cite{PSS+22} and Kelly, Müyesser and Pokrovskiy~\cite{KMP23}.

This sums up most of the proof ideas.
To obtain the improved bounds on $p$ in \cref{thm:main}\ref{item:main-balanced}, we have to dig a bit deeper into the technical machinery.
In particular, we use a refinement of the work of Kahn, Narayanan and Park~\cite{FKN+21} due to Espuny Díaz and Person~\cite{EP23}, which involves only a constant number of `planting steps' and thus avoids the factor $\log n$.

\section{Proofs of the applications}
\label{sec:applications-proofs}

In this section we derive the results stated in \cref{sec:applications} from our main result (\cref{thm:main}).
For convenience, we express some of the constant hierarchies in standard $\gg$-notation.
To be precise, we write $y \gg x$ to mean
that for any $y \in (0, 1]$, there exists an $x_0 \in (0,1)$
such that for all $x \leq x_0$ the subsequent statements
hold.  Hierarchies with more constants are defined in a
similar way and are to be read from left to right following the order that the constants are chosen.

\begin{remark}\label{rem:coupling}
	The random model of \cref{thm:main} is formulated in terms of directed hypergraphs, while our applications take place in the unordered setting.
	However, one can move from one random model to the other at negligible cost by a simple coupling argument.
	The details of this are spelled out in the proof of \cref{thm:main-simple}.
\end{remark}

\subsection{Perfect matchings}

To warm up, we derive \cref{thm:main-simple} from \cref{thm:main}.

\begin{proof}[Proof of \cref{thm:main-simple}]
	Set $\ell = 0$, $r=k$ and $s_1 = 4s_0$..
	We obtain $C>0$ from \cref{thm:main} with $s_0$ playing the role of $s_1$.
	Now suppose we are given $\eps > 0$.
	We show that for $n$ large enough,
	any $n$-vertex $k$-graph $G$ as in the statement and $p \geq k!C n^{-k+1} \log n$,
	a random sparsification $G' \sim G_p$, in which every edge is kept with probability $p$, contains a perfect matching with probability at least~$1-\eps$.

		{To this end, we define $H$ as the $(k,k)$-digraph on $V(G)$ whose ordered edges consist of the $k$-tuples obtained by ordering the vertices of each $k$-edge of $G$ in all possible ways.}
	Let $L$ be the $(k,0,k)$-link on vertex set $\{1,\dots,k\}$ consisting of a single edge $(1,\dots,k)$.
	Note that for any  $S \subset V(G)$, a perfect matching of $G[S]$ corresponds to $H[S]$ satisfying $\HamConn_L$.
	So the property $s$-graph $P = \PG{H}{\HamConn_L}{s}$ of $\HamConn_L$ in $H$ satisfies $\delta_{1}(P) \geq (1 - s^{-2}) \binom{n-1}{s - 1}$  for each $s_0 \leq s \leq 4 s_0$ which is divisible by $k$.
	Finally, it is easy to see that every $L$-chain is $(1/(k-1), 1/(k-1))$-balanced.
	We can therefore conclude by \cref{thm:main} that for $n$ large enough, a random sparsification $H' \sim H_{p/k!}$, in which every ordered edge is kept with probability $p/k!$, contains a closed Hamilton $L$-chain with probability at least $1-\eps$.

	Given an outcome $H' \sim H_{p/k!}$, denote by $G' \subset G$ the $k$-graph on $V(G)$ that contains a $k$-edge $\{v_1,\dots,v_k\}$ if $(v_1,\dots,v_n)$ is an ordered $k$-edge in $H'$.
	It follows that $G' \sim G_q$ with $q = 1-(1-{p/k!})^{k!} \leq p$.\COMMENT{Here we used that $(1+x)^r \geq 1+rx$ for $x \geq -1$ and $r \geq 1$.}
	Since a closed Hamilton $L$-chain in $H'$ corresponds to a perfect matching in~$G'$, we are done.
\end{proof}

\subsection{Perfect tilings}

Next we prove \cref{thm:perfect-tilings-minimum-degree}.
We remark that we do not need the full strength of \cref{thm:main} to show this.
In fact, it suffices to apply the special case for degree conditions of perfect matchings,
which was already known~\cite{KKK+22,PSS+22}, together with the work of Riordan~\cite{Rio22}.
Hence the purpose of this section is to illustrate how one transitions the abstract condition of \cref{thm:main} to Dirac-type results.

Minimum degree conditions are approximately inherited to typical induced subgraphs of constant order.
To formalise this end, we define $\DegF{d}{\delta}$, for any $0 \leq d < k$, as the family of all $k$-graphs~$G$ with $\delta_d(G) \geq \delta \binom{v(G)-d}{k-d}$.
The following result can be proven by applying concentration inequalities
and appears in various variations in the literature;
this specific version can be found in \cite[Lemma 4.9]{Lan23}.

\begin{lemma}[Minimum degree inheritance] \label{lem:inheritance-minimum-degree}
	Let $1/k,\,1/q,\,\mu \gg 1/s \gg 1/n$ and $\delta \geq 0$.
	Suppose $G$ is an $n$-vertex $k$-graph with $\delta_d(G) \geq (\delta + \mu) \tbinom{n-d}{k-d}$.
	Then $P:=\PG{G}{\DegF{d}{\delta+\mu/2}}{s}$ satisfies $\delta_{q}(P) \geq \left(1-e^{-\sqrt{s}} \right)  \tbinom{n-q}{s-q}$.
\end{lemma}

A similar result holds for uniformly dense graphs.
The next lemma appears in \cite[Lemma 9.1]{Lan23} (which also briefly surveys related work for other types of hypergraph quasirandomness).

\begin{lemma}[Uniform density inheritance]\label{lem:grabbing-uniformly-dense}
	Let $1/k,\,1/{q},\,d \gg \eps' \gg \eps,\, 1/s \gg 1/n$.
	Suppose $G$ is an $n$-vertex $k$-graph satisfying $\DenF{\eps}{d } $.
	Then $P=\PG{G}{\DenF{ \eps'}{d}}{s}$  satisfies  $\delta_{q}(P) \geq \left(1-e^{-\sqrt{s}} \right)  \tbinom{n-q}{s-q}$.
\end{lemma}

A straightforward application of \cref{thm:main} suffices to show that \cref{thm:perfect-tilings-minimum-degree} holds for $p = \Omega(n^{-1/{d_1}} \log n)$.
However, this is not quite what we are aiming for.
To obtain the correct fractional power of the logarithmic term, we combine \cref{thm:main-simple} with following result due to Riordan~\cite{Rio22}.

For a $k$-graph $F$ on $m$ vertices,
let $\ER{n}{q}{m}(F)$ be the random $m$-graph where (the vertex set of) each potential $F$-copy in $K_n^{(k)}$ is independently present with probability $q$ (there may be parallel edges, but this is irrelevant for us).
Riordan's result links the appearance of copies of $F$ in $\ER{n}{p}{k}$ to $\ER{n}{q}{m}(F)$ for appropriate $q$.

\begin{theorem}[{\cite[Theorem 18]{Rio22}}]\label{thm:rio}
	Let $F$ be a fixed strictly 1-balanced $k$-graph with $v(F)=m\geq 3$ and $e(F) = s$.
	Let $d_1 = s/(m -1)$.
	There are $\eps,C>0$ such that if $p = p(n) \leq n^{-1/d_1+\eps}$,
	then, for some $q = q(n) = (1+o(1))Cp^s$,
	we may couple $\ER{n}{p}{k}$ and  $\ER{n}{q}{m}(F)$   such that
	with high probability for every edge present in $\ER{n}{q}{m}(F)$ the corresponding copy of $F$ is present in $\ER{n}{p}{k}$.
\end{theorem}

Observe that Riordan's result not only applies to $\ER{n}{p}{k}$ but also to $G_p$ for any $n$-vertex $k$-graph $G$ and $G_{q}(F)$, where $G_{q}(F)$ refers to the sparsification with probability $q$ of the $m$-graph on $V(G)$ where each copy of $F$ in  $G$ is represented by an $m$-edge.
To see this, we view $G_p$ as a subgraph of $\ER{n}{p}{k}$ and~$G_{q}(F)$ as a subgraph of $\ER{n}{q}{m}(F)$.
Then given the coupling succeeds, an edge in $G_{q}(F)$ gives rise to a copy of $F$ in $G_p$.

\begin{proof}[Proof of Theorem~\ref{thm:perfect-tilings-minimum-degree}]
	Consider $1 \leq d \leq k-1$, $\mu>0$ and a $1$-balanced $k$-graph $F$ on $m$ vertices as in the statement.
	Set $\delta = \th_d^{}(\til_F)$.
	We choose $s_0$ divisible by $m$ sufficiently large to apply \cref{lem:inheritance-minimum-degree} with $q=1$ and sufficiently large with respect to $m$ and $\mu$.
	Let $C_1, s_1$ be obtained from \cref{thm:main-simple} with $m$ playing the role of $k$.
	Let $G$ be an $n$-vertex $k$-graph with $\delta_d(G)\geq (\delta+\mu)\binom{n-d}{k-d}$ with $n$ divisible by $m$.
	Let $H$ be the $m$-graph on $V(G)$ with an edge $Y$ of size $m$ whenever $G[Y]$ contains a copy of $F$.

	Fix $s_0 \leq s \leq s_1$ divisible by $m$, and set $P = \PG{H}{\mat_m}{s}$ to be the property $s$-graph of $H$ with respect to the property of having a perfect matching.
	We claim that $\delta_{1}(P) \geq (1-s^{-2}) \binom{n-1}{s - 1}$.
	To see this, let $Q=\PG{G}{\DegF{1}{\delta+\mu/2}}{s}$.
	So $Q$ is the property graph with an $s$-edge $S \subset V(G)$ whenever $\delta_d(G[S]) \geq (\delta + \mu/2) \binom{s-d}{k-d}$.
	By \cref{lem:inheritance-minimum-degree}, we have $\delta_1(Q) \geq (1-e^{-\sqrt{s}}) \binom{n-1}{s - 1} \geq (1-s^{-2}) \binom{n-1}{s - 1}$.
	Moreover, by the definition of $\delta = \th_d^{}(\til_F)$, the induced $k$-graph $G[S]$ has a perfect $F$-tiling for every edge of $Q$.
	So $H[S]$ has a perfect matching for every edge of $Q$.
	This gives $Q \subset P$, and the claim follows.

	Now let $q = C_1 n^{-m+1} \log n$.
	Since all assumptions are satisfied, we may apply \cref{thm:main-simple} to see that $H' \sim H_q$ contains a perfect matching with high probability.
	Put into the language of \cref{thm:rio}, we have that $G_{q}(F)$ contains a perfect matching with high probability.
	Let $p$ be defined by $q=C_2p^{e(F)}$, where $C_2$ is given by Theorem~\ref{thm:rio}.
	Observe that $p\leq n^{-1/d_1(F) + o(1)}$ by definition of $1$-density.
	Hence, we can couple $\Gs$ and $G_{q}(F)$ such that with high probability each edge in $G_{q}(F)$ yields a copy of $F$ in $G_p$, which completes the proof.
\end{proof}

The proof of \cref{thm:perfect-tilings-uniformly-dense} is almost identical to the proof of \cref{thm:perfect-tilings-minimum-degree}, and hence we omit the details.
The only difference is that we also use \cref{lem:grabbing-uniformly-dense} to verify that uniform density is inherited in a strong form.

\subsection{Loose cycles}\label{sec:loose-cycles-proofs}
In this section, we show \cref{thm:loose-hamilton-cycles-codegree,thm:loose-hamilton-cycles-l=1}.
Let us introduce a threshold for this purpose.
For a $k$-graph $G$ and $1\leq \ell<k$, we denote by $\partial_\ell(G)$ the set of $\ell$-sets in $V(G)$ which are contained in an edge of $G$.
Let us say that a $k$-graph $G$ is \emph{strongly Hamilton $\ell$-connected} if $\delta_\ell(G) > 0$ and for every pair of disjoint $\{v_1,\dots,v_\ell\},\,\{w_1,\dots,w_\ell\} \in \partial_\ell(G)$, there is a Hamilton $\ell$-path that starts with $(v_1,\dots,v_\ell)$ and ends with $(w_1,\dots,w_\ell)$.
We define $\th_d^{}(\SHC_{k,\ell})$ as the infimum over all $\delta \in [0,1]$ such that for all $\mu>0$ and large enough $n$,
every $n$-vertex $k$-graph $G$ with $n \equiv \ell \bmod (k-\ell)$ and $\delta_d(G) \geq (\delta +\mu) \binom{n-d}{k-d}$ is strongly Hamilton $\ell$-connected.

While strong Hamilton connectedness is in general considerably harder to obtain, it is known for the case $d=k-1$ that $\th_{k-1}^{}(\SHC_{k,\ell})=\th_{k-1}^{}(\ham_{k,\ell})$.
This can be proved by carrying out the same argument that was used to determine $\th_{k-1}^{}(\ham_{k,\ell})$~\cite{KMO10,RRS09a}.

\begin{theorem}\label{thm:hamilton-connctedness-codegree}
	We have $\th_{k-1}^{}(\SHC_{k,\ell}) = \th_{k-1}^{}(\ham_{k,\ell})$ for $2 \leq \ell \leq k-1$.
\end{theorem}

In the case of loose cycles, we can obtain the same for all degree types.
This was recently proved in~\cite[Theorem 9.3]{AKL+23}.

\begin{theorem}\label{thm:hamilton-connctedness-loose}
	We have $\th_d^{}(\SHC_{k,1}) = \th_d^{}(\ham_{k,1})$ for  $1\leq d \leq k-1$.
\end{theorem}

Given this, we are ready to show the main result of this section.
For a $k$-graph $G$, we denote by $\ori{C}(G)$ the $k$-digraph on $V(G)$ obtained by adding all possible orientations of every edge of $G$.

\begin{proof}[Proof of \cref{thm:loose-hamilton-cycles-codegree,thm:loose-hamilton-cycles-l=1}]
	Let $1\leq d, \ell \leq k-1$ and $\mu >0$.
	The hypothesis of \cref{thm:loose-hamilton-cycles-codegree,thm:loose-hamilton-cycles-l=1} are captured by the restrictions that either $d = k-1$ or $\ell = 1$ holds.
	Let $\delta = \th_{d}^{}(\ham_{k,\ell})$.
	
	Note that $k$-uniform $\ell$-cycles can be described as chains as follows.
	Let $r = k - \ell$, and let $L$ be the $(k,\ell,r)$-link on vertex set $\{1,\dots,k\}$ consisting of a single edge $(1,\dots,k)$.
	Note that a Hamilton $\ell$-cycle in a $k$-graph $G$ corresponds to a closed Hamilton $L$-chain in the $(k,\ell)$-digraph $H = \ori{C}(G) \cup \ori{C}(\partial_\ell G)$.
	Next, define $d^\ast = 1/(k-\ell)$ and $\lambda = \ell / (k-\ell)$.
	Let $A$ be an $n$-vertex $L$-chain.
	We claim that $A$ is $(d^\ast,\lambda)$-balanced.
	To this end, consider a subgraph $I$ in $A$ on $v$ vertices and $e > 0$ edges.
	Clearly, $I$ has at most $d^\ast v$ edges.
	Moreover, if $v \leq  {n}/{2v(L)}$, then $I$ cannot `wrap around'.
	Thus $v \geq (k - \ell)e + \ell$ (the $\ell$-path achieves equality), and this is the same as $e \leq (v - \ell)/(k - \ell)$.
	So in this case, we even have $e \leq d^\ast v - \lambda$.
	This shows that $A$ is indeed $(d^\ast,\lambda)$-balanced.
	
	Now we fix the required constants.
	Set $q=2\ell+1$.
	With $k, \mu, q, \delta$ as input,
	\cref{lem:inheritance-minimum-degree} yields $s_0$.
	\cref{thm:hamilton-connctedness-codegree,thm:hamilton-connctedness-loose} together imply that there exists $s'_0 \geq s_0$ sufficiently large such that for all $s \geq s'_0$ with  $s \equiv \ell \bmod k-\ell$,
	all $k$-graphs $G$ with $\delta_d(G) \geq (\delta + \mu/2) \binom{s-d}{k-d}$ are strongly Hamilton $\ell$-connected
	and $1-e^{-\sqrt{s}}> 1-s^{-2}$.
	Fix any $s_1 \geq \max\{ s'_0, 5k \}$ with $s_1 \equiv \ell \bmod k-\ell$.
	For the parameters $r, \ell, k, s_1$, \cref{thm:main} outputs $C > 0$.
	Choose $n$ divisible by $k - \ell$ and sufficiently large in terms of all previous parameters.
	
	Let $G$ be an $n$-vertex $k$-graph with $\delta_d(G)\geq (\delta+\mu) \binom{n-d}{k-d}$,
	and let $H = \ori{C}(G) \cup \ori{C}(\partial_\ell G)$.
	To apply \Cref{thm:main}, we will check that $\PG{H}{\HC_L}{s}$ satisfies the required degree conditions, for all $s_1 \leq s \leq 4 s_1$ with $s \equiv \ell \bmod k-\ell$.
	
	Indeed, fix such an $s$ and let $P = \PG{H}{\HC_L}{s}$.
	We claim that $\delta_{q}(P) \geq (1 - s^{-2}) \binom{n-q}{s - q}$.
	To see this, let $Q=\PG{G}{\DegF{d}{\delta+\mu/2}}{s}$, so $Q$ is the property graph with an $s$-edge $S \subset V(G)$ whenever $\delta_{d}(G[S]) \geq (\delta + \mu/2) \binom{s-d}{k-d}$.
	By \cref{lem:inheritance-minimum-degree}, we have $\delta_{q}(Q) \geq (1-e^{-\sqrt{s}}) \binom{s-d}{k-d} \geq (1-s^{-2}) \binom{s-d}{k-d}$.
	Moreover, by our choice of constants,
	the induced $k$-graph $G[S]$ is strongly Hamilton $\ell$-connected.
	Since a Hamilton $\ell$-path in $G[S]$ corresponds to a Hamilton $L$-chain in $H[S]$, it follows that $H[S]$ satisfies $\HC_L$.
	Thus $Q \subseteq P$, and therefore $\delta_q(P) \geq \delta_q(Q)$, from which we obtain the desired inequality.
	
	Now we are ready to apply \cref{thm:main}.
	Our goal, in each considered case, is to show that $H' \sim H_p^{(k)}$ contains a closed Hamilton $L$-chain with high probability.
	In the case $\ell = 1$, $L$-chains are $(d^\ast, \lambda)$-balanced with $d^\ast = \lambda = 1/(k-1)$.
	Thus we set $p = C n^{-1/d^\ast} \log n = C n^{-k+1}\log n$ and apply \cref{thm:main}\ref{item:main-general} to deduce that $H' \sim H_p^{(k)}$ contains a closed Hamilton $L$-chain with high probability.
	On the other hand, in the case $\ell \geq  2$, we have that $L$-chains are $(d^\ast, \lambda)$-balanced with $d^\ast = 1/(k-\ell) < \ell/(k-\ell) = \lambda$.
	Thus we set $p = \omega(n^{-1/d^\ast}) = \omega(n^{-k+\ell})$ and apply \cref{thm:main}\ref{item:main-balanced} to deduce that $H' \sim H_p^{(k)}$ contains a closed Hamilton $L$-chain with high probability.
	
	When then finish with the same coupling argument as in the proof of \cref{thm:main-simple} (see also \cref{rem:coupling}), to find that $G' \sim G_{p'}$, for $p'$ within a constant factor of $p$, contains a Hamilton $\ell$-cycle with high probability, as desired.
\end{proof}

\subsection{From Hamilton frameworks to connectedness}
\label{sec:frameworks}
We begin by describing $(t-k+1)$th powers of tight cycles as $L$-chains for $2 \leq k \leq t$.
Let $L$ be the $(k,t-1,1)$-link on vertex set $\{1,\dots,t\}$ with an edge $(1,i_1,\dots,i_{k-1})$ for every choice of pairwise distinct $2 \leq  i_1 < \dots < i_{k-1} \leq t$.
Note that the $(t-k+1)$th power of a tight Hamilton cycle in $G$ corresponds to a closed Hamilton $L$-chain.
We denote $\HamConn_{k,t} = \HamConn_L$ for this choice of $L$.

We derive the following general theorem about (powers of) tight Hamilton cycles from our main result.te
For a family of $k$-graphs $\cP$ and $\delta=1-1/s^2$, a $k$-graph $G$ \emph{$(q,s)$-robustly satisfies} $\cP$ if the minimum $q$-degree of the property $s$-graph $\PG{G}{  \cP   }{s}$ is at least $\delta \tbinom{n-q}{s-q}$.

\begin{theorem}\label{thm:framework-connectedness}
	For each $2 \leq k \leq t$ with $3 \leq t$, $d = \binom{t-1}{k-1}$, there is an $s_0$ such the that following holds for $s\geq s_0$ and $p = \omega(1/d)$.
	Let $\cP$ be a family of $s$-vertex $k$-graphs such that the family $\{K_t(R)\colon R \in \cP\}$ admits a Hamilton framework.
	Let $G$ be a $k$-graph on $n$ vertices that satisfies $(q,s)$-robustly $\cP$.
	Then $G' \sim \RS{G}{p}{(k)}$ contains the $(t-k+1)$st power of a tight Hamilton cycle with high probability.
\end{theorem}

For the proof, we require a result by Lang and Sanhueza-Matamala~\cite[Theorem 5.8]{LS24a}, which translates robust Hamilton frameworks into robust connectedness.

\begin{theorem}\label{thm:framework-connectedness-robust}
	Let $1/t,1/q \geq {1/s_1} \gg 1/s_2 \geq 1/s_3 \gg 1/n$.
	Let $\cP$ be a family of $s_1$-vertex $t$-graphs that admits a Hamilton framework.
	Let $K$ be a $t$-graph on $n$ vertices that satisfies $(2t,s_1)$-robustly $\cP$.
	Then there is an $n$-vertex $(t,t-1)$-digraph $K' \subset \ori{C}(K) \cup \ori{C}(\partial K)$ that satisfies $(q,s)$-robustly $\HamConn_{t,t}$ for every $s_2 \leq s \leq s_3$.
\end{theorem}

\begin{proof}[Proof of \cref{thm:framework-connectedness}]
	Introduce $s_1, s_2,s_3$ with $1/t \geq {1/s_1} \gg 1/s_2 \gg 1/s_3 \gg 1/n$ as in \cref{thm:framework-connectedness-robust}.
	Set $\ell = t-1$, $r=1$ and $q= 2\ell+1$.
	Let $\cP$ be a family of $s$-vertex $k$-graphs such that the family $\cQ=\{K_t(G)\colon G \in \cP\}$ admits a Hamilton framework.
	Let $G$ be a $k$-graph on $n$ vertices that satisfies $(q,s_1)$-robustly $\cP$.
	Note that $K = K_t(G)$ satisfies $(q,s_1)$-robustly $\cQ$.
	By \cref{thm:framework-connectedness-robust} there is an $n$-vertex $(t,t-1)$-digraph $K' \subset \ori{C}(K) \cup \ori{C}(\partial K)$ that satisfies $(q,s)$-robustly $\HamConn_{t,t}$ for every $s_2 \leq s \leq 4s_2$.
	From this we obtain an $n$-vertex $(k,t-1)$-digraph $H \subset \ori{C}(G) \cup \ori{G}(\partial G)$ that satisfies $(q,s)$-robustly $\HamConn_{k,t}$ for every $s_2 \geq s \geq 4s_2$.
	(For every $t$-tuple $X$ in $K'$, add to $H$ every $t$-tuple $Y$ that can be obtained by removing $t-k$ entries from $X$.)
	
	Let $L$ be the $(k,\ell,r)$-link on vertex set $\{1,\dots,t\}$ with an edge $(1,i_1,\dots,i_{k-1})$ for every choice of pairwise distinct $2 \leq  i_1 < \dots < i_{k-1} \leq t$.
	Let $d = \binom{t-1}{k-1}$ and $\lambda = d+1$.
	Let $A$ be an $L$-chain on $n$ vertices.
	We claim that $A$ is $(d,\lambda)$-balanced.
	To this end, consider a subgraph $I$ in $A$ on $v$ vertices and $e > 0$ edges.
	Clearly, $I$ has at most $dv$ edges.
	Now suppose that $v \leq  {n}/{2v(L)}$.
	So in particular, $I$ cannot `wrap around'.
	It follows that $I$ is missing all edges of one copy of $L$, and, because $t\geq 3$, at least one edge of another copy of $L$.
	This gives $e \leq d v - \lambda$.
	This shows that $A$ is indeed $(d,\lambda)$-balanced.

	Now set $p = \omega(n^{-1/d})$, and apply \cref{thm:main}\ref{item:main-balanced} to deduce that $H' \sim H_p^{(k)}$ contains a closed Hamilton $L$-chain with high probability.
	When then finish with the same coupling argument as in the proof of \cref{thm:main-simple} (see also \cref{rem:coupling}), to find that $G' \sim G_{p'}$, for $p'$ within a constant factor of $p$, contains a Hamilton $\ell$-cycle with high probability, as desired.
\end{proof}

We are now ready to derive \cref{thm:tight-hamilton-cycles-d-degree-framework}.

\begin{proof}[Proof of \cref{thm:tight-hamilton-cycles-d-degree-framework}]
	Set $\delta = \delta_d^{\hf}(k,t)$.
	Given $\eps > 0$, set $q=2k$, and choose $s,c$ with $1/q, \eps \gg 1/s \gg c \gg 1/n$.
	Let $G$ be a $k$-graph on $n$ vertices with $\delta_d(G) \geq (\delta + \eps) \binom{n}{2}$.
	By \cref{lem:inheritance-minimum-degree} the property $s$-graph $P =\PG{G}{\DegF{d}{\delta+\eps/2}}{s}$ satisfies $\delta_{q}(P) \geq  (1-1/s^2)  \tbinom{n-q}{s-q}$.
	Let $\cP$ be the family of $s$-vertex $t$-graphs $K_t(R)$ with $R \in \DegF{d}{\delta+\eps/2}$.
	By definition of $\delta$, the family $\cP$ admits a Hamilton framework.
	Thus we can finish by applying \cref{thm:framework-connectedness}.
\end{proof}

Lastly, we come to the proof of \cref{thm:power-hamiton-cycles-degree-sequence}.
We define the graph class $\DegSeq_{t,\mu}$ as all $n$-vertex graphs~$G$ with degree sequences $d_1,\dots,d_n$ such that $d_i > (t-2)n/t + i + \mu n$ for every $i \leq n/t$.
It was recently established that large enough graphs in $\DegSeq_{t,\mu}$ contain $(t-1)$-powers of Hamilton cycles~\cite{LS23,ST17}.
(And this is sharp as constructions show.)
Using standard probabilistic tools, it is not hard to see that $\DegSeq$ is approximately closed under taking small induced subgraphs~\cite[Lemma 12.6]{LS24a}.

\begin{lemma}[Degree sequence inheritance] \label{lem:inheritance-degree-sequence}
	Let $1/t,\eps \gg 1/s \gg 1/n$.
	Suppose $G$ is an $n$-vertex graph in $\DegSeq_{t,\eps}$.
	Then $P=\PG{G}{\DegSeq_{t,\eps/2}}{s}$ satisfies $\delta_{2k}(P) \geq \left(1- e^{-\sqrt{s}} \right)  \tbinom{n-2t}{s-2t}$.
\end{lemma}

Moreover, it follows from prior arguments~\cite{LS23,Tre15} that the family $\DegSeq$ admits a Hamilton framework~\cite[Corollary 12.8]{LS24a}.

\begin{proposition}\label{pro:degree-sequence-hamilton-framework}
	Given $t\geq 2$ and $\eps >0$, let $n$ be sufficiently large.
	Then the family of $t$-graphs $K=K_t(G)$ on $n \geq n_0$ vertices with $G \in \DegSeq_{t,\eps}$ admits a Hamilton framework $F$ with $F(K) = K$.
\end{proposition}

Given these two results, the proof of \cref{thm:power-hamiton-cycles-degree-sequence} can be derived from \cref{thm:framework-connectedness} along the same lines as \cref{thm:power-hamiton-cycles-minimum-degree}.
We omit the details.

\section{Proof of the main result}
\label{sec:proof-main-result}

In the following, we formulate four intermediate results (\cref{lem:typical-to-correct-distribution,lemma:fromcorrecttospread,theorem:FKNP,lemma:fromstrongspreadtothreshold}) that together combine to \cref{thm:main}.
The proofs of \cref{lem:typical-to-correct-distribution,lemma:fromcorrecttospread,lemma:fromstrongspreadtothreshold} follow in \cref{sec:correct-to-spread-proof,sec:typical-to-correct-distribution-proof,section:strongspread-to-threshold}.

\subsection{Correct distributions}
The first lemma tells us that,  under  the assumptions of \cref{thm:main}, there is a probability distribution over Hamilton $L$-chains in $G$ which satisfies the `correct' properties for our purposes.

The \emph{$2$-shadow} $\partial_2 H$ of a (non-directed) hypergraph $H$ is the graph on vertex set $V(H)$ where $uv$ forms an edge if there exists an edge $e$ in $H$ such that $\{u,v\} \subseteq e$.
A subgraph $H' \subseteq H$ is a \emph{component} of $H$
if $\partial_2 H'$ forms a non-empty component in $\partial_2 H$, and $H'$ is edge-maximal with that property.
Thus the number of components of $H$ is the number of components of $\partial_2H$ that contain at least one edge, and in particular, an edgeless hypergraph $H$ has zero components.
For digraphs we make the same definitions by applying them to their corresponding undirected hypergraph (ignoring the orientation of every edge).

Let $L$ be a $(k, \ell, r)$-link.
For a positive integer $n$, let $\cC(n,L)$ be the set of all closed $L$-chains with vertex set $[n]$.
Let $\mu$ be a probability measure on $\cC(n,L)$ and let $C\in \cC(n,L)$ be chosen according to $\mu$.
For $K>0$,
we say that $\mu$ is
\emph{$K$-correct} if \[ \probability[I\subseteq C] \leq K^{v(I)}n^{c-v(I)}\] for all directed hypergraphs $I$ with vertex set in $[n]$
and $c$ components.
Finally, a $k$-digraph $G$ with vertex set $[n]$ \emph{supports} $\mu$ if $\mu(C')=0$ for all $C'$ that are not a subgraph of $G$.

Now we are ready to state the first lemma for the proof of \cref{lem:typical-to-correct-distribution}.
Its proof is deferred to \cref{sec:typical-to-correct-distribution-proof}.

\begin{lemma}[Correct distribution] \label{lem:typical-to-correct-distribution}
	For every $r, \ell, k, s_1 \geq 0$ with $s_1\geq 5(r+\ell)$ and $q=2 \ell + 1$, there is $K>0$ such that the following holds for all sufficiently large $n$ which are divisible by~$r$.
	Let $L$ be a $(k, \ell, r)$-link and let $G$ be an $n$-vertex $(k, \ell)$-digraph.
	Suppose that for every $s_1 \leq s \leq 4 s_1$ with $s \equiv \ell \bmod r$, the $s$-graph $P = \PG{G}{\HamConn_L}{s}$ satisfies $\delta_{q}(P) \geq (1 - s^{-2}) \binom{n-q}{s - q}$.
	Then there exists a $K$-correct distribution $\mu$ on $\mathcal{C}(n, L)$ that is supported by $G^{(k)}$.
\end{lemma}

\subsection{From correct to spread distributions}

Given the assumptions of \cref{thm:main}, we can apply Lemma~\ref{lem:typical-to-correct-distribution} to obtain a correct distribution.
In our next step towards proving \cref{thm:main}, we show that correct distributions meet the `spreadness' conditions that are required to apply the work of Frankston, Kahn, Narayanan and Park~\cite{FKN+21} as well as Espuny Díaz and Person~\cite{EP23}.
We remark that in the cases considered in \cref{thm:main}\ref{item:main-balanced}, we need a stronger `spreadness' property, which  allows us to remove the logarithmic factor from the probability threshold in the last step of the proof of \cref{thm:main}.

Suppose that $\mathcal{H}$ is a family of subsets of a finite ground set $Z$.
A probability distribution $\mu$ supported on $\mathcal{H}$ is \emph{$q$-spread} if by choosing $C \in \mathcal{H}$ randomly according to $\mu$, we have, for any $I \subseteq Z$, that
\[ \probability[ I \subseteq C ] \leq q^{|I|}.\]

{We also say that $\mu$ is \emph{$(q,a,b)$-strongly-spread}
if $\mu$ is $q$-spread and if for all
$S$ contained in some set of $\cH$
and all $a|S| \leq j \leq |S|$, we have
\begin{align*}
	\sum_{I\subseteq S,\, |I|=j} \probability[ I \subseteq C ] \leq (bq)^{j},
\end{align*}
where again $C$ is chosen at random according to $\mu$.}

Note that \cref{lem:typical-to-correct-distribution} returns a distribution $\mu$ supported on $\cC(n, L)$.
We translate this into a spread distribution via the following auxiliary hypergraphs.
Let $\cH(n,L)$ be the hypergraph with vertex $Z = [n]^k$ and edge set $\{E(C)\colon C\in \cC(n,L)\}$.
Let $\cH(n,L,\mu) \subset\cH(n,L)$ contain the edges $E(C)$ of $\cH(n,L)$ such that $\mu(C) > 0$.
Thus $\mu$ corresponds naturally to a distribution $\mu'$ over $\cH(n,L,\mu)$ with ground set $Z = [n]^k$.
We say that $\mu$ on $\cC(n, L)$ is \emph{$q$-spread} if the distribution $\mu'$ on $\cH(n,L,\mu)$ is $q$-spread.
We define \emph{$(q, a, b)$-strongly-spread} distributions $\mu$ on $\cC(n, L)$ analogously.

The next lemma states that correct distributions of appropriately balanced closed chains are spread.
Its proof is given in \cref{sec:correct-to-spread-proof}.

\begin{lemma}[Spread distribution] \label{lemma:fromcorrecttospread}
	For every $r, \ell, k, K \geq 0$, there is $K'>0$ such that the following holds for all sufficiently large $n$.
	Let $d,\lambda>0$ and let $L$ be a $(k, \ell, r)$-link such that sufficiently large closed $L$-chains are $(d, \lambda)$-balanced.
	Suppose $\mu$ is a $K$-correct probability distribution on $\cC(n,L)$.
	Then we have the following.
	\begin{enumerate}[\upshape{(\roman*)}]
		\item \label{item:correcttospread-superspread} If $\lambda=d$,
		      then $\mu$ is $K'n^{-1/d}$-spread.
		\item \label{item:correcttospread-strongspread} If $\lambda > d$,
		      then $\mu$ is $(K'n^{-1/d},(e(L)n/r)^{-(\lambda/d - 1)/2}, K')$-strongly-spread.
	\end{enumerate}
\end{lemma}

\subsection{From spread distributions to thresholds}

Given the assumptions of \cref{thm:main}, we can combine \cref{lem:typical-to-correct-distribution,lemma:fromcorrecttospread} to obtain a suitably spread distribution.
In our final step, we turn this into the desired statement about probability thresholds.

To obtain part~\ref{item:main-general}, we employ the celebrated result of Frankston, Kahn, Narayanan and Park~\cite{FKN+21}.
The following result was essentially proved in \cite[Theorem 1.6]{FKN+21} using the (equivalent) language of `spread multihypergraphs', the version we cite uses spread distributions and appears in the work of Pham, Sah, Sawney and Simkin~\cite[Theorem 1.2]{PSS+22}.
For a finite set $Z$ and $p \in [0,1]$, we denote by $Z(p)$ the binomial distribution on subsets of $Z$ where each vertex is present with probability $p$.

\begin{theorem} \label{theorem:FKNP}
	There exists $C > 0$ such that the following holds.
	Let $Z$ be a non-empty ground set and $\mathcal{H}$ be a non-empty collection of subsets of $Z$.
	Suppose that there is a $q$-spread probability distribution supported on $\mathcal{H}$.
	Then, for $p = \min\{ 1, C q \log |Z| \}$, the set $Z(p)$ contains an element of $\mathcal{H}$ with probability tending to $1$ as $|Z| \to \infty$.
\end{theorem}

We cannot use \cref{theorem:FKNP} to prove part~\ref{item:main-balanced} of \Cref{thm:main} due to the undesired additional logarithmic factor.
Since this factor cannot be removed in general, we require the following result, which is tailored to the setting of $L$-chains.
The proof is given in \cref{section:strongspread-to-threshold}.

\begin{lemma}[Threshold Lemma] \label{lemma:fromstrongspreadtothreshold}
	Let $K', d, \lambda, k, \ell, r, \alpha$ with $  d > \lambda$ and $\alpha > 0$ be given.
	Let $d,\lambda>0$ and let $L$ be a $(k, \ell, r)$-link such that sufficiently large closed $L$-chains are $(d, \lambda)$-balanced.
	Let $\mu$ be a probability distribution supported on $\mathcal{C}(n, L)$, which is $(K' n^{-1/d},(e(L)n/r)^{-\alpha},K')$-strongly-spread.
	Then, for $q = \omega(n^{-1/d})$, a $q$-binomial random subset of $[n]^k$ contains an element of $\cH(n,L,\mu)$ as a subset with high probability.
\end{lemma}

\subsection{Putting everything together}

To conclude this section, we derive our main result from \Cref{lem:typical-to-correct-distribution,lemma:fromcorrecttospread,lemma:fromstrongspreadtothreshold,theorem:FKNP}.

\begin{proof}[Proof of \Cref{thm:main}]
	Given $k,r, \ell$ and $s_1 \geq 5(r+\ell)$, let $K$ be large enough to apply \cref{lem:typical-to-correct-distribution}.
	Choose $K'$ to satisfy the requirements of \cref{lemma:fromcorrecttospread}.
	In the following, we assume that is $n$ sufficiently large to apply \cref{lem:typical-to-correct-distribution,lemma:fromcorrecttospread}.
	Consider $L$ and $G$ as in the statement of \cref{thm:main}.
	By \cref{lem:typical-to-correct-distribution}, there exists a $K$-correct distribution $\mu$ on $\mathcal{C}(n, L)$ that is supported by $G^{(k)}$.

	To obtain part~\ref{item:main-general} of the theorem, let us assume that $\lambda=d$.
	In this case, $\mu$ is $K'n^{-1/d}$-spread by \cref{lemma:fromcorrecttospread}.
	We may then finish by applying \Cref{theorem:FKNP} with $Z = [n]^k$ and $\mathcal{H} = \cH(n,L,\mu)$.
	For part~\ref{item:main-balanced}, we assume that $\lambda>d$.
	So $\mu$ is $(K'n^{-1/d},(e(L)n/r)^{-(\lambda/d - 1)/2},K')$-strongly-spread by \cref{lemma:fromcorrecttospread}.
	Note that $\alpha := (\lambda/d - 1)/2 >0$.
	So we can finish by applying \Cref{lemma:fromstrongspreadtothreshold}.
\end{proof}

\section{Proof of \cref{lem:typical-to-correct-distribution}} \label{sec:typical-to-correct-distribution-proof}

Next we prove \cref{lem:typical-to-correct-distribution}, which is the most involved part of this article.
Recall the setting of \cref{lem:typical-to-correct-distribution}:
for a given $(k, \ell, r)$-link $L$ and $n$-vertex $(k, \ell)$-digraph $G$ on vertex set $[n]$, we need to find a distribution $\mu$ over $\mathcal{C}(L, n)$,
the set of spanning closed $L$-chains on $[n]$, such that $G^{(k)}$ supports $\mu$.
The distribution must be $K$-correct, meaning that if $C$ is drawn according to $\mu$, then $\probability[I \subseteq C] \leq K^v n^{c-v}$ for all $v$-vertex subgraphs $I$ of $G^{(k)}$ with $c$ components.

Our distribution $\mu$ is described by a randomised algorithm which outputs a spanning $L$-chain $H$ in $G$.
Given $s_1$ and $n$ divisible by $r$,
we find integers $s_2, m$ with $s_2$ satisfying $s_1 \leq s_2 \leq 4 s_1$ such that $n = (m-1)s_1+s_2+m(s_1-2\ell)$.
Let $n_1 = s_1(m-1)+s_2$ and $n_2 = m(s_1-2\ell)$.

The algorithm builds $H$ in three phases.
First, we partition the vertex set of $G$ into two sets $V_1, V_2$ of size $n_1, n_2$, respectively.
Secondly, in $G[V_1]$ we find a random spanning collection of vertex-disjoint open $L$-chains $R_1, \dotsc, R_{m}$, where $m-1$ of the chains use $s_1$ vertices and the remaining one uses $s_2$ vertices.
Thirdly, in $V_2$ we find a random collection of vertex-disjoint subgraphs $Q_1, \dotsc, Q_{m}$, of size $s_1 - 2 \ell$ each, such that the concatenation $R_1 Q_1 R_2 Q_2 \dotsb R_t Q_t$ is a closed Hamilton $L$-chain in~$G$.

We choose the subgraphs $R_i, Q_i$ by looking at the corresponding `property graphs' $P_j = \PG{G}{\HamConn_L}{s_j}$ of~$G$, for $1\leq j \leq 2$.
To simplify the discussion, assume $s_1 = s_2$ and $s_1 \equiv \ell \bmod r$ in this paragraph.
Since each edge $X \in P_1[V_1]$ yields naturally an $s_1$-vertex open $L$-chain in $G[V_1]$,
a perfect matching $M \subseteq P_1[V_1]$ yields a collection of vertex-disjoint open $L$-chains on $s_1$ vertices each which can serve as our desired set of open $L$-chains $R_1, \dotsc, R_m$.
Our assumptions imply that $\delta_{2 \ell + 1}(P_1)$ is sufficiently large, and this will in turn imply that not only $P_1[V_1]$ will admit perfect matchings but also a spread distribution over its set of perfect matchings.
In this way, we are able to choose a matching $M \subseteq P_1[V_1]$ at random,
which in turn yields a random collection of vertex-disjoint open $L$-chains,
as needed in the description of the algorithm.
The choice of $Q_i$ is done in a similar fashion, after restricting ourselves to consider matchings $M_2 \subseteq P_2[V_2]$ which are suitably aligned with the open ends of the already-selected open chains $R_i$.
We then argue that these collections of paths satisfy the required conditions.

The first tool we require to analyse the previous construction is the following.
It says that if an $s$-uniform matching $M$ is chosen from a suitably spread random distribution, then $\partial_2 M$ contains each graph `correctly'.

\begin{lemma}[Spread Matching to Correct Distribution]\label{lem:spread-property-matching-to-correct-distribution}
	For all $C>1$ the following holds.
	Let $n, s$ be positive integers.
	Let $P$ be an $n$-vertex $s$-graph, and
	suppose that $\mu$ is a $(Cn^{-s+1})$-spread distribution on the set of perfect matchings in $P$.
	Let $I$ be a $v$-vertex $2$-graph with $c \leq n/2$ components such that $V(I) \subseteq V(P)$.
	Let $M$ be a random perfect matching in $P$ drawn according to $\mu$.
	Then $\probability[ I \subset \partial_2 M] \leq (2C)^v n^{c-v}$.
\end{lemma}

\begin{proof}
	A component $F$ of $I$ can appear as a subgraph of $\partial_2 M$ only if there exists an edge $e$ of $M \subseteq P$ with $V(F) \subseteq e$.
	We first estimate the number of (not necessarily perfect) matchings $N=\{e_1,\ldots,e_r\}$ such that each $e_i$ contains at least one component of $I$.
	Suppose there are $r$ edges in $N$ that cover the $c$ components in $I$.
	There are at most $\binom{c}{r}$ ways to select the first component for each edge and then for each such choice at most further $r^{c-r}$ ways to assign the remaining components to the edges of $N$.
	If $v_i$ refers to the total number of vertices of all components in $I$ that are contained in $e_i$,
	then there are at most $n^{s-v_i}$ $s$-edges in $P$ that contain the components in $e_i$;
	in total, there are at most $n^{rs-v}$ choices for $N$ given the assignment of the components to the $e_i$.
	Therefore, there are at most $\sum_{r=1}^c \binom{c}{r}r^{c-r}n^{sr-v}$ choices for $N$.
	We compute
	\begin{align*}
		\probability[ I \subset \partial_2 M]
		 & \leq \sum_{r=1}^c \sum_{N}\probability[N \subseteq M]
		\leq \sum_{r=1}^c \binom{c}{r}r^{c-r}n^{sr-v} \cdot (Cn^{-s+1})^r \\
		 & \leq C^vn^{-v}	\sum_{r=1}^c \binom{c}{r}r^{c-r}n^{r}
		\leq C^vn^{c-v} \cdot 2^c
		\leq (2C)^vn^{c-v},
	\end{align*}
	which completes the proof.
\end{proof}

We also use the following variation of \cref{lem:spread-property-matching-to-correct-distribution} where we work with different uniformities and we pick one edge of different size.
We need this more technical result because in general we work with $s_1 \neq s_2$.

\begin{corollary}\label{lem:spread-property-matching-to-correct-distribution-2}
	For all $C,s_2 \geq 1$,
	there exists $K$ such that the following holds.
	Suppose $n, m, s_1\in \bN$ with $n = s_2 + s_1 m$.
	Let $P$ be an $n$-vertex $s_1$-graph such that for any $s_2$-set $X\subseteq V(P)$,
	there exists a $(Cn^{-s_1 + 1})$-spread distribution on the set of perfect matchings in $P - X$.
	Suppose $\mu_2$ is a distribution over $\binom{V(P)}{s_2}$ such that if $X$ is drawn according to $\mu_2$ and $X'$ is a fixed $s_2$-set, then $\probability[X = X'] \leq 2 \binom{n}{s_2}^{-1}$.

	We choose a random $2$-graph $R$ in $V(P)$ as follows.
	First, pick a random $s_2$-set $X \subseteq V(P)$ according to $\mu_2$,
	and in $P - X$ choose a matching $M \subseteq P-X$ according to a $(Cn^{-s_1 + 1})$-spread distribution.
	Let $R = \partial_2 ( X \cup M )$ (so $R$ is a $2$-graph consisting of one clique of size $s_2$ and $m$ cliques of size $s_1$).
	Let $I$ be a $v$-vertex $2$-graph with $c \leq n/2$ components such that $V(I) \subseteq V(P)$.
	Then $\probability[ I \subset R] \leq K^v n^{c-v}$.
\end{corollary}
\begin{proof}
	We proceed similarly as in the proof of  \cref{lem:spread-property-matching-to-correct-distribution}.
	A component $F$ of $I$ can appear in $R$ only if some of its components are completely contained in $\partial_2 X$
	and all other components are contained in $\partial_2 M$.
	Suppose that $C_1, \dotsc, C_c$ are the components of $I$.
	For $J \subseteq [c]$, define $C_J := \bigcup_{j \in J} C_j$.
	Given $J \subseteq [c]$,
	we estimate first the probability that $X$ contains precisely the components in $C_J$ and $\partial_2 M$ contains the other components of $I$.
	This is used then to estimate the desired probability using an union bound.

	Let $J \subseteq [c]$ be given.
	Suppose $C_J$ has $v_J$ vertices and define $D_J=I-C_J$.
	Suppose first that $v_J \leq s_2$.
	The set $V(C_J)$ is contained in $\binom{n - v_J}{s_2 - v_J}$ sets of size $s_2$,
	and each $s_2$-set is selected by $\mu_2$ with probability at most $2 \binom{n}{s_2}^{-1}$.
	Thus
	\[ \probability[ C_J \subseteq \partial_2 X ] \leq \frac{2 \binom{n - v_J}{s_2 - v_J}}{\binom{n}{s_2}}
		= \frac{2 \binom{s_2}{v_J}}{\binom{n}{v_J}}
		\leq 2 s_2^{v_J} n^{-v_J}, \]
	where we used $\frac{s_2-i}{n-i} \leq \frac{s_2}{n}$ for all $0\leq i \leq v_J-1$ in the previous step.
	If $s_2 > v_J$, then $\probability[ C_J \subseteq X ] = 0$, so the above bound is valid for any choice of $J$.

	By \cref{lem:spread-property-matching-to-correct-distribution},
	we obtain $\probability[D_J \subseteq \partial_2 M | C_J \subseteq  \partial_2X ] \leq (2C)^{v - v_J} n^{c - |J| - (v - v_J)}$.
	Set $K = 4 \max\{ C, s_2 \}$.
	This in turn yields
	$\probability[ C_J \subseteq\partial_2 X , D_J \subseteq \partial_2 M ]
		\leq \probability[C_J \subseteq\partial_2 X]\probability[D_J \subseteq \partial_2 M | C_J \subseteq\partial_2 X ]
		\leq (K/2)^{v} n^{c - v - |J|}$.

	Thus an union bound over all the choices of $J$ gives
	\begin{align*}
		\probability[ I \subseteq R ]
		 & \leq \sum_{J \subseteq [c]} \probability[ C_J \subseteq\partial_2 X , D_J \subseteq \partial_2 M ]      \\
		 & \leq (K/2)^v n^{c-v} \sum_{J \subseteq [c]} n^{-|J|} = (K/2)^v n^{c-v} \sum_{i=0}^c \binom{c}{i} n^{-i} \\
		 & = (K/2)^v n^{c-v} (1+n^{-1})^c \leq (K/2)^v n^{c-v} e^{c/n} \leq K^v n^{c-v},
	\end{align*}
	where in the last line we used $1+x \leq e^x$ and $c \leq n/2$ to obtain the (very crude) bounds $e^{c/v} \leq e^{1/2} \leq 2^v$.
\end{proof}

Let $G$ be a $k$-graph.
We say that a family of subgraphs $\cF$ is \emph{$q$-spread in $G$} if there is a $q$-spread distribution $\mu\colon \cF \to [0,1]$ with the ground set $E(G)$.
In order to apply \cref{lem:spread-property-matching-to-correct-distribution},
we need to ensure that a graph of large enough minimum degree contains a spread perfect matching.
Consider a partition $\cV$ of the vertex set of $G$.
In this context, we say that a $k$-set $S$ is \emph{partite} if it contains at most one vertex of each part of $\cV$.
We say that $G$ is \emph{$k$-partite} if all of its edges are partite with respect to some partition of $V(G)$ into $k$ parts.

The next lemma states that balanced $k$-partite $k$-graphs with large minimum $1$-degree contain spread perfect matchings.
This is known to be true in the non-partite setting~\cite{KKK+22,KMP23,PSS+22} and the available proofs can be adapted to work in the partite setting.

\begin{lemma}[Minimum Degree to Spread Perfect Matching]\label{lem:minimum-degree-to-spread-partite-perfect-matching}
	Let $1/k\gg 1/C \gg 1/m$ and $n=km$.
	Let $G$ be a $k$-partite $k$-graph with each part of size $m$ and $\delta_1(G) \geq (1-1/(2k))m^{k-1}$.
	Then the set of perfect matchings in $G$ is $(Cn^{-k+1})$-spread.
\end{lemma}

We also record a consequence of the simple fact that the uniform distribution over the set of all Hamilton cycles in the complete graph $K_n$ is $((1+o(1))e/n)$-spread (see \cite[Section 2]{KNP21} for a similar calculation).

\begin{lemma}\label{lemma:spreadofhamcycles}
	Let $n$ be sufficiently large, and let $C \subseteq K_n$ be a Hamilton cycle chosen uniformly at random.
	Then $\probability[I \subseteq C] \leq (2en^{-1})^{|I|}$ for any $I \subseteq E(K_n)$.
\end{lemma}

Now we are ready to give the proof of the main result of this section.

\begin{proof}[Proof of \cref{lem:typical-to-correct-distribution}]
	Given $k, r, \ell, s_1$, we introduce $K', K'',K, n$ such that
	\begin{align*}
		1/s_1 \gg 1/K'\gg 1/K'' \gg 1/K \gg 1/n.
	\end{align*}
	Suppose $n$ is divisible by $r$.
	We may assume that $s_1 \equiv \ell \bmod r$, as otherwise we simply increase $s_1$ until the assumption holds and the assumptions of the lemma are still valid for all $s_1\leq s \leq 3s_1$, which is sufficient for us.
	Let $m, m'$ be integers such that
	\begin{equation}
		n = 2m(s_1 - \ell)+m', \label{equation:n}
	\end{equation}
	with $m \geq 0$ and $0 \leq m' < 2(s_1 - \ell)$.
	We define $s_2 = s_1 + m'$.
	Note that $s_1 \leq s_2 \leq 3 s_1$ holds.
	Moreover, it is easy to check that $s_2 \equiv \ell \bmod r$.
	Indeed, since $n \equiv 0 \bmod r$ and $s_1 - \ell \equiv 0 \bmod r$, from \eqref{equation:n} we obtain that $m' = n - 2m(s_1 - \ell) \equiv 0 \bmod r$, and therefore $s_2 = s_1 + m' \equiv s_1 \bmod r$.

	Now consider a $(k,\ell,r)$-link $L$, and let $G$ be an $n$-vertex $k$-digraph with vertex set $[n]$ satisfying the assumptions of \cref{lem:typical-to-correct-distribution}.
	Hence for each $s \in \{s_1, s_2\}$, the bound $\delta_{2 \ell + 1}(P) \geq (1 - 1/s^2) \binom{n-(2\ell+1)}{s - (2\ell+1)}$ holds for $P = \PG{G}{\HamConn_L}{s}$.
	Our goal is to show that there exists a $K$-correct distribution $\mu$ over $\cC(n, L)$ such that $G^{(k)}$ supports $\mu$.
	In the following three steps, we construct a random closed Hamilton $L$-chain $H$ in $G^{(k)}$.
	In the fourth step we show that the distribution over $L$-chains obtained from this construction is indeed $K$-correct.

	\subsubsection*{Step 1: Partitioning the vertex set.}
	Note that
	\[ n = s_1(m-1)+s_2 + m(s_1 - 2\ell). \]
	To begin, partition $V(G)$ deterministically into sets $V_{1}$ and $V_{2}$ such that
	\begin{align*}
		|V_{1}|  = s_1 (m-1) + s_2 \quad\text{and}\quad
		|V_{2}|  = (s_1 - 2 \ell)m
	\end{align*}
	and such that further conditions hold which we specify below.
	The existence of such a partition can be verified by a standard probabilistic argument.

	For each $1\leq i \leq 2$, set $P_i = \PG{G}{\HamConn_L}{s_i}$.
	The first step of our random procedure whose distribution defines $\mu$ is that we uniformly at random choose an $s_2$-set $V_0 \in P_1[V_1]$ (hence $G[V_0]$ contains an $s_2$-vertex (open) $L$-chain).
	Given this (random) choice of $V_0$,
	we deterministically partition $V_{1} \setminus V_0$ into sets $V_1^1,\dots,V_1^{s_1}$ of size $m-1$ each,
	and we partition $V_{2}$ into sets $V_2^1,\dots,V_2^{s_1-2\ell}$ of size $m$ each
	such that certain properties hold (again that such partitions exist can be verified by an probabilistic argument).

	Let $P^1_1 \subseteq P_1[V_1 \setminus V_0]$ be the $s_1$-partite $s_1$-graph containing the edges which have non-empty intersection with all $V^1_{1},\dots,V_1^{s_1}$.
	Let $P^2_1 \subset P_1$ be the $s_1$-graph that contains all edges $J \cup \{v_1,\dots,v_{s_1-2\ell}\}$ where $J \subset V_1$ is a $2\ell$-set and $v_i \in V_2^i$ for each $1\leq i \leq s_1-2\ell$.

	Now we describe the properties.
	\begin{enumerate}[\upshape (D1)]
		\item \label{item:degp3} $\delta_1(P_2[V_1]) \geq (1 - 1/(3s_2)) \binom{|V_1|-1}{s_2 - 1}$.
		\item \label{item:degp1} $\delta_1(P^1_1) \geq (1-1/(3s_1))(m-1)^{s_1-1}$.
		\item  \label{item:degp2} For all $2\ell$-sets $J \subset V_1$ and $v \in V_2$, there are at least $(1-1/(3s_1)) m^{s_1-2\ell-1}$ edges in $P^2_1$ that contain $J \cup \{v\}$.
	\end{enumerate}
	Recall that $\delta_{2\ell + 1}(P_i) \geq (1-{1}/{s_i^2})  \binom{n-(2\ell + 1)}{s_i-(2\ell + 1)}$ holds for each $1\leq i \leq 2$.
	Therefore a  standard concentration analysis shows that there are such choices for $V_1,V_2$ and the partitions $V^1_{1},\dots,V_1^{s_1}$ and $V_2^1,\dots,V_2^{s_1-2\ell}$ as claimed.

	\subsubsection*{Step 2: Covering $V_1$.}
	In the following, we cover the vertices of $G[V_{1}]$ with $m-1$ open $L$-chains of order $s_1$ and one open $L$-chain of order $s_2$.
	Let $\cH_1$ be the set of perfect matchings in $P^1_1$.
	By \cref{lem:minimum-degree-to-spread-partite-perfect-matching} and property~\ref{item:degp1}, there is a distribution $\mu_1\colon \cH_1 \to [0,1]$ that is 	$(K'm^{-s_1+1})$-spread.
	Let  $M_1$ be a random perfect matching chosen according to $\mu_1$.
	For each edge $X$ of $M_1$, we fix an $s_1$-vertex open $L$-chain in $G[X]$, which is possible since $M_1 \subseteq P^1_1 \subset \PG{G}{\HamConn_L}{s_1}$.
	Similarly, we fix an $s_2$-vertex open $L$-chain in $G[V_0]$, which is possible since by construction, $V_0$ is an edge in $P_2[V_1] \subseteq \PG{G}{\HamConn_L}{s_2}$.
	Let $R$ be the collection of the $m$ $L$-chains that are obtained in this way.

	\subsubsection*{Step 3: Connections.}
	We now turn $R$ into a Hamilton $L$-chain.
	Let $R_1,\dots,R_{m} \in R$ denote a random permutation of its chains.
	For each $1 \leq i \leq m$, let $R^-_i$ and $R^+_i$ be the initial and terminal $\ell$ vertices of the $L$-chain $R_i$, respectively (note that $R^-_i \cap R^+_i = \emptyset$ as $s_1 \geq 2 \ell$).
	Next, we specify a set of pairwise disjoint $L$-chains $C_1,\dots,C_m$ such that $C_i$ connects $R^+_i$ with $R^-_{i+1}$ (index computation modulo $m$),
	the paths are vertex-disjoint from $V_1$ except for $\bigcup_{i \in [m]} (R^-_i \cup R^+_i)$, and such that $V_2$ is completely covered by $\bigcup_{i \in [m]} V(C_i)$.

	We find $C_1, \dotsc, C_m$ as follows.
	For each $1 \leq i \leq m$, let $T_i = R^+_i \cup R^-_{i+1}$ (index computation modulo $m$).
	Hence each $T_i$ has exactly $2 \ell$ vertices.
	Set $s_1''=(s_1-2\ell+1)$, and let $P_1''$ be an auxiliary $s_1''$-partite $s_1''$-graph with vertex partition $V_{2}^1,\dots,V_2^{s_1-2\ell},\{T_1,\dots,T_m\}$ and an edge $\{v_1,\dots,v_{s_1-2\ell} , J\}$ whenever $\{v_1,\dots,v_{s_1-2\ell}\} \cup J$ forms an edge in $P^2_1$.
	By property~\ref{item:degp2}, it follows that $\delta_1(P_1'') \geq (1-1/(2s_1'')) m^{s_1''-1}$.
	Let $\cH_2$ denote the set of perfect matchings of $P_1''$.
	By \cref{lem:minimum-degree-to-spread-partite-perfect-matching}, there is a distribution $\mu_2 \colon \cH_2 \to [0,1]$ that is $(K'm^{-s_1''+1})$-spread.

	Let $V'_2 = V_2 \cup \bigcup_{i \in [m]} T_i$.
	Note that each perfect matching of $P_1''$ naturally corresponds to a unique perfect matching $M_2$ in $P_1[V'_2]$ such that $T_i$ is contained in a different edge of $M_2$ for each $1\leq i \leq m$.
	We choose such a matching $M_2$ randomly according to~$\mu_2$.
	By definition of $\PG{G}{\HamConn_L}{s_1}$, we can select an $s_1$-vertex open $L$-chain $C_i$ in $G[X_i]$ which goes from $R^+_i$ to $R^-_{i+1}$ for each $1 \leq i \leq m$.
	Let $C = \bigcup_{i=1}^m C_i$.
	This choice of $C_1, \dotsc, C_m$ joins the chains $R_1, \dotsc, R_m$ and together they form a closed Hamilton $L$-chain $H$, as required.

	\subsubsection*{Step 4: Verifying correctness.}
	Let $\mu$ be the distribution over $\mathcal{H}$ (the family of subgraphs of $G$ which are closed Hamilton $L$-chains) corresponding to the random choice of $H \in \mathcal{H}$ given by this procedure.
	To finish, we need to show that $\mu$ is correct.
	To see this, consider a $v$-vertex subgraph $I' \subset G^{(k)}$ with $c$ components.
	Let $H \in \cH$ be drawn according to $\mu$.
	We have to show that $\probability[I' \subset H] \leq { K^{v}} n^{c-v}$.
	Let $I = \partial_2 I'$ be the $2$-shadow of $I'$.
	Note that $\probability[I' \subset H]\leq \probability[I \subset \partial_2 H]$.
	Moreover, $I$ has also $v$ vertices and $c$ components, and each component of $I'$ corresponds naturally to one in $I$.
	So it suffices to show that $\probability[I \subset \partial_2 H] \leq { K^{v}} n^{c-v}$.

	For each $1\leq j \leq 2$, set $I_j = I[V_j]$ and write $v_j,c_j$ for the number of vertices and components in $I_j$.
	Let $c^\ast_2$ be the number of \emph{floating} components, that is, those components of $I$ whose vertices are completely contained in $V_2$.
	Set $c^\ast_1=c-c^\ast_2$.
	Denote the components of $I$ by $J_1,\dots,J_c$.
	We call a component $X$ in $I_1 \cup I_2$ a \emph{$J$-minicomponent}, if $X$ is contained in a component $J$ of $I$.
	We also say just \emph{minicomponent} if it is a $J$-minicomponent for some $J$.
	For each $1 \leq j \leq c$, let $d_j$ denote the number of $J_j$-minicomponents in $I_1$.
	We record here that
	\begin{align}
		c_1      & = \sum_{j\in [c]} d_j, \label{equation:c1}                    \\
		c_2^\ast & = |\{j \colon d_j = 0\}|, \text{ and} \label{equation:c2star} \\
		c_1^\ast & = |\{j \colon d_j \neq 0\}| \label{equation:c1star}.
	\end{align}

	Clearly, if $I \subseteq \partial_2 H$, then also $I_1$ must be contained in $\partial_2 H$, which means that every minicomponent in $I_1$ must be covered by some chain $R_i$ as chosen in Step 1.
	But more is true: the minicomponents in $I_1$ must be covered in a restricted way so they can be covered correctly with the Hamilton $L$-chain later.

	There are two main restrictions, which we call the `adjacency obstacle' and the `cyclic obstacle'.
	We give examples to illustrate these obstacles, beginning with the adjacency obstacle.
	Suppose that $J$ is a component of $I$ such that there are three $J$-minicomponents $X_1, X_2, X_3$ in $I_1$, and one $J$-minicomponent $Y$ in $I_2$.
	If $J \subseteq \partial_2 H$ holds, then $Y$ is located in a single $L$-chain $C_i$, as chosen in Step 3.
	This implies that $Y$ has neighbours only in $R_{i} \cup R_{i+1}$ (at most two consecutive chains of $R$).
	Thus, it cannot happen that $X_1, X_2, X_3$ lie in three different chains $R_i, R_j, R_k$ of $R$ in Step 1.

	To illustrate the `cyclic obstruction', consider $J$-minicomponents $X_1, X_2, X_3$ in $I_1$ and $J$-minicomponents $Y_1, Y_2, Y_3$ in $I_2$, where $Y_1$ has neighbours in $X_1$ and $X_2$, $Y_2$ has neighbours in $X_2$ and $X_3$, and $Y_3$ has neighbours in $X_3$ and $X_1$.
	Suppose that $X_1, X_2, X_3$ are covered by the distinct chains $R_a, R_b, R_c$ respectively.
	Even though this avoids the adjacency obstacle, if $m > 3$ there is no way to order the components $R_a, R_b, R_c$ in such a way that $Y_1, Y_2, Y_3$ can be consistently located in chains $C_j$.
	This is because we would need the pairs $R_a R_b$, $R_b R_c$, and $R_c R_a$ to appear consecutively along the order chosen in Step 2, which is not possible.
	Thus the components need to be covered by the chains $R_i$ such that a cyclic arrangement of those chains is possible.

	To capture all of these restrictions, we introduce two auxiliary graphs, which take care of the adjacency and cyclic obstructions respectively.
	First, we use an auxiliary bipartite graph to record the connections between the minicomponents of $I_1$ and $I_2$.
	Let $\Gamma$ be a bipartite graph whose two vertex classes correspond to the sets of all the minicomponents in $I_1, I_2$ respectively.
	Then, for minicomponents $X\in I_1, Y\in I_2$, we add the edge $XY$ if there exists $x \in X, y \in Y$ such that $xy \in E(\partial_2 I)$.
	We say that a choice of $R$ is \emph{adjacency-valid} if every $X \in I_1$ is covered by a chain of $R$, and for every $Y \in I_2$, all the components in $N_\Gamma(Y)$ are covered by at most two chains $R_i, R_j$.

	Given an adjacency-valid choice of $R$, let $\cR=\{R_1,\ldots,R_m\}$ be the set of chains which form~$R$.
	We define an auxiliary graph $\Xi$ with $V(\Xi) \subseteq \cR$ as follows.
	The vertex set $V(\Xi)$ corresponds precisely to those chains in $\cR$ which contain at least one minicomponent of $I_1$.
	The edges are defined as follows.
	For each minicomponent $Y \in I_2$, the minicomponents in $N_\Gamma(Y)$ are covered by at most two components of $R$, say $R_a, R_b$.
	If those components are distinct, we add the edge $R_a R_b$ to $\Xi$.
	We say that $R$ is \emph{cyclic-valid} if $\Xi$ can be extended to a Hamilton cycle on the vertex set $\cR$.
	Equivalently, either $\Xi$ is a Hamilton cycle or is a vertex-disjoint collection of paths plus isolated vertices.
	We say that the choice of $R$ is \emph{valid} if it is both adjacency-valid and cyclic-valid.

	As explained before, if $I \subseteq \partial_2 H$ holds, then the choice of $R$ is valid.
	We define the following event:
	\begin{enumerate}[label={(E\arabic*)}]
		\item $\Eval$ is the event that the choice of $R$ is valid.
	\end{enumerate}

	The next condition we need is that the ordering of the chains $R_1, \dotsc, R_m$ of $\cR$ (chosen at the beginning of Step 3) is suitable for the connecting step.
	If $R$ is valid and the graph $\Xi$ has a path $P$, then we need that the vertices of $P$ appear consecutively along the random cyclic order we choose in this step.
	Similarly, if $\Xi$ is a Hamilton cycle, then we need that the order follows precisely the ordering dictated by the cycle.
	We say that the order $R_1, \dotsc, R_m$ of the chains is \emph{consistent} if
	for each $Y \in I_2$, the two chains $R_a, R_b$ of $R$ which contain all minicomponents of $N_\Gamma(Y)$ appear consecutively in the given order.
	By the previous discussion, $I \subseteq \partial_2 H$ implies that the order is consistent.
	We define that
	\begin{enumerate}[resume, label={(E\arabic*)}]
		\item  \label{itm:event-consistent}$\Econ$ is the event that the order $R_1, \dotsc, R_m$ is consistent.
	\end{enumerate}

	Next, assuming that both $\Eval, \Econ$ hold, we consider what needs to happen in the connecting step to ensure that $I \subseteq \partial_2 H$ holds.
	We clearly need that $I_2 \subseteq \partial_2 H$, but this is not enough, because we need to ensure that the non-floating minicomponents of $I_2$ are contained in specific $L$-chains, not just in any $L$-chain.
	To be precise, recall that $C_1,\dots,C_m$ are the $L$-chains that connect each $R_i$ with $R_{i+1}$ chosen according to $\mu_2$.
	Assuming $\Eval \cap \Econ$, we
	note that for each minicomponent $Y$ of $I_2$ with $N_\Gamma(Y) \neq \emptyset$, we have that the minicomponents in $I_1$ which are adjacent to $Y$ are covered either by one chain $R_i$, or by two consecutive chains $R_{i}, R_{i+1}$.
	In the first case $Y$ must be covered by $C_{i-1}$ or $C_i$, and in the second case $Y$ must be covered by $C_i$ (if this does not happen then $I \subseteq \partial_2 H$ cannot be true).
	If this holds, we say that the minicomponent $Y$ is \emph{anchored}.
	We define
	\begin{enumerate}[resume, label={(E\arabic*)}]
		\item \label{itm:event-anchored} $\Eanc$ is the event that $I_2 \subseteq \partial_2 C$, and also each minicomponent $Y$ in $I_2$ with $N_\Gamma(Y) \neq \emptyset$ is anchored.
	\end{enumerate}

	By the previous discussion, we have that
	\begin{equation}
		\probability[ I \subseteq \partial_2 R ]
		\leq \probability[ \Eval \cap \Econ \cap \Eanc ]
		= \probability[\Econ, \Eval] \probability[ \Eanc \mid \Eval, \Econ].
		\label{equation:keyinequality-correctness}
	\end{equation}
	In what comes next, we estimate the value of the probabilities appearing in the last expression.

	To begin, we consider the simpler case of a single component $J$ in $I$.
	Given $J$-minicomponents $X_1, \dotsc, X_{d} \in I_1$, a partition $\mathcal{P}$ of $\{ X_1, \dotsc, X_{d} \}$ is \emph{valid} if there exists a valid choice of $R$ such that each partition class in $\mathcal{P}$ corresponds to a set of minicomponents which is covered by the same chain $R \in \cR$.
	We give a simple counting argument which bounds the number of valid partitions in this situation.

	\begin{claim} \label{claim:onecomponent}
		Suppose $J$ is a component of $I$, and there are $d$ many $J$-minicomponents in $I_1$.
		Then there are at most $3^{d-1}$ valid partitions for that set of minicomponents.
	\end{claim}

	\begin{proofclaim}
		We describe a greedy procedure which can generate all possible valid partitions and then estimate in how many ways it can be done.

		We construct sets $P_i$, for each $1 \leq i \leq m$, where initially all of them are empty.
		Begin by picking an arbitrary minicomponent in $I_1$ and adding it to $P_1$.
		Next, if there are any $J$-minicomponents in $I_1$ not yet allocated to a set of the partition, we proceed as follows.
		Since $J$ is connected, all $J$-minicomponents belong to the same component in the auxiliary graph $\Gamma$.
		Hence, there exist $J$-minicomponents $X_1, X_2 \in I_1$ such that $X_1$ has already been allocated to a set $P_i$, $X_2$ has not been allocated to any set, and $X_1, X_2$ have a common neighbour $Y \in I_2$ in $\Gamma$, which is a $J$-minicomponent in $I_2$.
		Since the partition needs to be valid, if $X_1, X_2$ are allocated to different sets in our partition then it would yield an edge in the graph $\Xi$.
		Since we can only have paths and cycles in $\Xi$,
		then this means that there are at most three possibilities to allocate $X_2$.
		Thus, we allocate $X_2$, and iterate.
		Once every minicomponent belongs to some set, the partition is found.

		Clearly all valid partitions can be obtained in this way and there are at most $3^{d - 1}$ possible valid choices in our allocation.
	\end{proofclaim}

	Now we consider ways to partition all minicomponents in $I_1$ at once.
	Given a valid choice of $R$, let $\kappa$ be the number of components of the auxiliary graph $\Xi$.
	Recall that $V(\Xi)$ corresponds precisely to those chains in $\cR$ which house at least one minicomponent of $I_1$.
	There are exactly $m$ chains in $\cR$, so $\kappa \leq m$ holds.
	If there is at least one minicomponent in $I_1$ then $\kappa > 0$ as well.
	Given the set $\cX$ of all minicomponents in $I_1$ let $P_N$ be the number of valid partitions of $\cX$ which yield exactly $N$ components in the graph $\Xi$.
	In the next claim, we bound $P_N$.

	\begin{claim} \label{claim:manycomponents}
		For any $0 \leq N \leq m$, we have
		$P_N \leq 14^{v_1} n^{c^\ast_1 - N}$.
	\end{claim}

	\begin{proofclaim}
		We note first that if there are no minicomponents in $I_1$, we have $P_0 = 1$, otherwise $P_0 = 0$.
		From now on, we suppose $N \geq 1$.
		To bound $P_N$, we need to estimate how many valid choices for $R$ yield exactly $N$ components in $\Xi$.
		Every valid choice for $R$ can be obtained by specifying sets $G_1,\ldots,G_N$ such that $G_i$ corresponds to the set of all minicomponents that belong to the $i$th path component in $\Xi$.
		For all (non-floating) components $J$ of $I$, all $J$-minicomponents in $I_1$ belong to a single component in $\Xi$,
		and hence we first assign the non-floating components $J$ to some $G_i$ such that no $G_i$ is empty.
		Let $g_i$ be the number of $J$ assigned to $G_i$.
		Afterwards allocate the $J$-minicomponents within one component.
		Note that we overcount here by a factor of $N!$, as any permutation among the $G_i$ gives the same set of components.

		We continue with the analysis of this procedure.
		There are exactly $\binom{c^\ast_1 - 1}{N - 1} \leq 2^{c_1^\ast}\leq  2^{v_1}$ choices for $(g_1,\ldots, g_N)$; that is, where $g_i\geq 1$ and $\sum_{i=1}^Ng_i=c_1^\ast$.
		Given the $g_i$,
		we assign the $c^\ast_1$ non-floating components of $I$ to $G_1, \dotsc, G_N$.
		(This can be done in $\binom{c^\ast_1}{g_1} \binom{c^\ast_1-g_1}{g_2} \dots   \binom{g_N}{g_N}$ ways.)
		In total, the number of possibilities are
		\begin{align*}
			\binom{c^\ast_1 - 1}{N - 1} \binom{c^\ast_1}{g_1} \binom{c^\ast_1-g_1}{g_2} \dots  \binom{g_N}{ {g_N}}
			 & \leq 2^{v_1} \frac{c^\ast_1!}{\prod_{i=1}^N g_i!}.
		\end{align*}

		Assuming we have groups $G_1, \dotsc, G_N$ already defined, now we need to count how many ways there are to allocate the components of each $G_i$ into a connected component of~$\Xi$.

		For each (non-floating) component $J$ of $I$, we have to allocate the $J$-minicomponents in $I_1$ to the different chains $R_i$.
		This is a consecutive set among the chains $R_i$ and \cref{claim:onecomponent} tells us that there are at most $3^d$ partitions that yield such a consecutive set allocation, where~$d$ is the number of $J$-minicomponents in $I_1$.
		Hence altogether for all $J$, there at most $3^{v_1}$ such choices.

		Now we focus on each component of $\Xi$ individually, say the one which is generated by $G_i$.
		We have to appropriately glue together the different (linear) allocations of the $J$-minicomponents for all $J\in G_i$.
		Let $d_i$ be the number of all $J$-minicomponents for $J\in G_i$.
		We start with one allocation and then for each following allocation, we have at most $2d_i$ choices how the glue it together with the already merged allocations; that is, there are at most $(2d_i)^{g_i-1}$ possible joint allocations for this component of $\Xi$.
		Hence, given $G_1,\ldots,G_N$, there are at most $2^{v_1} \prod_{i=1}^N d_{i}^{g_i - 1}$ allocations of the minicomponents such that $\Xi$ has $N$ components.

		For fixed reals $\lambda_1, \dotsc, \lambda_N\geq 0$ with $\sum_{i=1}^N \lambda_i = \lambda$ and reals $x_1, \dotsc, x_N\geq 0$ with $\sum x_i \leq x$,
		the product $\prod_{i=1}^N x_i^{\lambda_i}$ is maximised when $x_i = x \lambda_i / \lambda$.
		Observe that $\sum_{i=1}^N (g_i - 1) = c^\ast_1 - N$ and $\sum_{i=1}^N d_{i} \leq n$.
		Hence
		\begin{align*}
			3^{v_1}\cdot 2^{v_1} \cdot \prod_{i=1}^N d_{i}^{g_i - 1}
			 & \leq 6^{v_1} \frac{n^{c^\ast_1 - N}}{(c^\ast_1 - N)^{c^\ast_1-N}}  \prod_{i\colon g_i \geq 1} (g_i - 1)^{g_i - 1} \\
			 & \leq 7^{v_1} \frac{n^{c^\ast_1 - N}}{(c^\ast_1 - N)!}  \prod_{i=1}^N g_i!.
		\end{align*}

		Therefore, we conclude that
		(recall the overcount by a factor of $N!$)
		\begin{align*}
			P_N
			\leq
			\frac{1}{N!}
			\cdot 7^{v_1} \frac{c^\ast_1!}{\prod_{i=1}^N g_i!} \cdot \frac{n^{c^\ast_1 - N}}{(c^\ast_1 - N)!}  \prod_{i=1}^N g_i!
			\leq \binom{c_1^\ast}{N}7^{v_1} n^{c^\ast_1 - N}
			\leq 14^{v_1} n^{c^\ast_1 - N}
		\end{align*}
		as desired.
	\end{proofclaim}

	Now, assuming a valid choice of $R$, let $0 \leq \beta \leq m$ be the number of vertices in $\Xi$.
	Recall that $\kappa$ refers to the number of components of $\Xi$, and clearly $0 \leq \kappa \leq \beta$.
	We can now bound the probability that $\beta, \kappa$ attain a certain value and simultaneously $\Eval$ holds.

	\begin{claim} \label{claim:e1}
		For all $0 \leq N \leq b \leq m$, we have $\probability[ \Eval , \beta = b, \kappa = N] \leq K''^{2v_1} n^{c^\ast_1 - N + b - v_1}$.
	\end{claim}

	\begin{proofclaim}
		Suppose a valid choice of $R$ yields a graph $\Xi$ with $N$ components.
		Such a choice of $R$ yields a partition of the set of all minicomponents in $I_1$ as in \Cref{claim:manycomponents},
		and there at most $14^{v_1} n^{c^\ast_1 - N}$ such valid partitions.
		Assume now that a valid partition $\cP$ of the set of all minicomponents in $I_1$ is given.
		Assuming that $|\cP| = b$, it suffices to show that $\Eval$ holds with probability at most $(5K''s_1)^{v_1} n^{b - v_1}$, as then the result will follow from a union bound (over all the possibilities of choosing the partition $\cP$).

		Recall that $R$ consists of $m-1$ paths on $s_1$ vertices each and one path of $s_2$ vertices each.
		Again, by replacing each path $R_i \in R$ with a clique $R^\ast_i$ on the same number of vertices,
		we can reduce the problem to the estimation of the probability that $I_1 \subseteq R^\ast$, where $R^\ast = R^\ast_1 \cup \dotsb \cup R^\ast_m$,
		subject to the condition that for each $P \in \cP$, all minicomponents in $P$ are covered by the same clique $R^\ast_i$.
		To encode the allocations, we introduce an auxiliary graph called $I'_1$.
		For each $P \in \cP$, add all edges between the different minicomponents $X \in P$ to obtain a supergraph $I'_1$ with exactly one component for each $P \in \cP$.
		In this way, we obtain a graph $I'_1$ with precisely $b$ components and $v(I'_1) = v(I_1) = v_1$ vertices.
		Note that if $I'_1 \subseteq R^\ast$ holds, then for each $P \in \cP$, all minicomponents in $P$ belong to the same clique $R^\ast_i$, as required.
		So now it suffices to study $\probability[I'_1 \subseteq R^\ast]$.

		To do this, we can apply \Cref{lem:spread-property-matching-to-correct-distribution-2} with  $s_1, |V_1|, I'_1, v_1, b, K'$ playing the roles of $s_1, n, I, v, c, C$.
		Indeed, by \ref{item:degp3} we have that $|P_2[V_1]| \geq (1 - 1/(3s_2)) \binom{|V_1|}{s_2}$,
		and since $V_0^1$ is chosen at random among all $s_2$-sets in $P_2[V_1]$,
		we have that each specific $s_2$-set has probability at most $|P_2[V_1]|^{-1} \leq 2 \binom{|V_1|}{s_2}^{-1}$ of being chosen, as required.
		Further, recall that the matching $M_1$ consisting of the $s_1$-sized paths of $R$ is chosen according to a $(K' m^{-s_1 + 1})$-spread distribution.
		Then \Cref{lem:spread-property-matching-to-correct-distribution-2} implies the existence of $K''$, depending on $K'$ and $s_1$ only, such that
		\[\probability[I'_1 \subseteq R^\ast] \leq (K'')^{v_1} |V_1|^{b - v_1}. \]
		Now, recall that $s_2 \leq 3 s_1$ holds.
		Therefore, from $n = s_1(m-1)+s_2+(s_1 - 2 \ell)m$, we obtain that $m \geq n/(5 s_1)$.
		This implies that $|V_1| = s_1(m-1)+s_2 \geq m \geq n / (5 s_1)$.
		Together with $v_1 \geq c_1$, this gives $|V_1|^{b - v_1} \leq  n^{b - v_1} (5 s_1)^{v_1}$.
		Plugging this in the inequality above, we have
		\[  \probability[I'_1 \subseteq R^\ast] \leq (K'')^{v_1} |V_1|^{b - v_1} \leq (5K''s_1)^{v_1} n^{b - v_1}, \]
		which completes the proof of the claim.
	\end{proofclaim}

	Next, we bound the probability of the event $\Econ$ as defined in~\ref{itm:event-consistent}.

	\begin{claim} \label{claim:ec}
		For all $0 \leq N \leq b \leq m$, we have $\probability[\Econ \mid \Eval, \beta = b, \kappa = N] \leq (30s_1)^{v_1} n^{-b + N}$.
	\end{claim}

	\begin{proofclaim}
		We capture the notion of ``consistency'' using the auxiliary graph $\Xi$.
		By assumption, the graph $\Xi$ has $N$ components and $b$ vertices in total, and is a subgraph of a Hamilton cycle on $m$ vertices, so either $\Xi$ itself is a Hamilton cycle, or $\Xi$ is a vertex-disjoint union of paths and isolated vertices.
		Hence $e(\Xi) \geq b - N$.

		A random ordering of the components corresponds to the choice of a uniform random Hamilton cycle $C$ on the vertex set $[m]$,
		and the event that the random ordering is consistent corresponds to the event that $C$ contains $\Xi$.

		By \cref{lemma:spreadofhamcycles}, we have $\probability[\Xi \subset C] \leq (2em^{-1})^{e(\Xi)}\leq (2em^{-1})^{b - N}$.
		Using that $m \geq n/(5s_1)$ and that $b - N \leq v_1$, we obtain that
		\[ \probability[\Econ \mid \Eval, \beta = b, \kappa = N] \leq (2em^{-1})^{b - N} \leq (30s_1)^{v_1} n^{-b+N}, \]
		as required.
	\end{proofclaim}

	\begin{claim} \label{claim:ece1}
		$\probability[ \Econ, \Eval] \leq K^{v_1} n^{c_1^\ast - v_1 }$.
	\end{claim}

	\begin{proofclaim}
		To estimate $\probability[ \Econ, \Eval]$, we can partition the event according to the numbers $\beta,\kappa$ of vertices and components  of the auxiliary graph $\Xi$.
		Note that in any valid choice of $R$, the auxiliary graph $\Xi$ has at most $v_1$ components and spans at most $v_1$ vertices.
		Hence, we can write
		\begin{align*}
			\probability[ \Econ, \Eval]
			 & = \sum_{b = 0}^{v_1} \sum_{N = 0}^{v_1} \probability[ \Econ, \Eval, \beta = b, \kappa = N]                                               \\
			 & = \sum_{b = 0}^{v_1} \sum_{N = 0}^{v_1} \probability[ \Econ  | \Eval, \beta = b, \kappa = N] \probability[\Eval, \beta = b, \kappa = N ] \\
			 & \leq \sum_{b = 0}^{v_1} \sum_{N = 0}^{v_1} K''^{3v_1} n^{c^\ast_1 - v_1}
			= (v_1+1)^2 K''^{3v_1} n^{c^\ast_1 - v_1},
		\end{align*}
		where in the inequality we used the two previous claims.
		As $K^{v_1}\geq (v_1+1)^2 K''^{3v_1}$ the claim follows.
	\end{proofclaim}

	Finally, we bound the probability of the event $\Eanc$ as defined in~\ref{itm:event-anchored}.

	\begin{claim} \label{claim:e2}
		$\probability[ \Eanc \mid \Eval, \Econ] \leq K^{v_2} n^{c_2^\ast- v_2 }$.
	\end{claim}

	\begin{proofclaim}
		{If $\Eanc$ holds, then for each non-floating minicomponent $Y$ in $I_2$, there exists $1 \leq i \leq m$ such that $Y$ is covered by either $C_{i-1}$ or $C_i$ (modulo $m$).
		Say that the index $j \in \{i-1, i\}$ such that $Y \subseteq C_j$ is the \emph{anchor index} of $Y$.
		There are at most $2^{c_2} \leq 2^{v_2}$ many possibilities for the choices of all anchor indexes.
		We plan to use a union bound over all possible choices of anchor indexes, so assume that such a choice is given, and we estimate the probability that each non-floating minicomponent $Y$ in $I_2$ is contained precisely in $C_i$, where $i$ is the anchor index of $Y$.}
		
		Consider an auxiliary graph $I'_2 \subseteq V'_2$ defined as follows.
		Recall that $V'_2$ is the disjoint union of $V_2$ and $\{T_1, \dotsc, T_m\}$.
		If $J \in I_2$ is a floating component, then $J$ corresponds naturally to a component in $I'_2$.
		Also, for each $1 \leq i \leq m$,
		we add to $I'_2$ the graph obtained by joining the (vertex) $T_i$ with an edge to each vertex of every non-floating component $J \in I_2$ whose anchor index is $i$.
		If $i$ is not the anchor index of any component, we do not add (the isolated vertex) $T_i$ to $I_2'$.
		Let $a$ be the number of integers $i$ that are an anchor index of some component.
		Thus we obtain that $I'_2$ has $c^\ast_2 + a$ components, which in total span $v_2 + a$ vertices.
		Thus (assuming $\Eval \cap \Econ$), we obtain that $\Eanc$ holds only if $I'_2 \subseteq \partial M_2$, where $M_2$ is the perfect matching chosen in the connecting step.

		By \cref{lem:spread-property-matching-to-correct-distribution} applied with $C, v(P''_1), P''_1, s''_1, I'_2$ playing the roles of $C, n, P, s, I$, we obtain
		\[ \probability[ \Eanc \mid \Eval, \Econ ] = \probability[ I'_2 \subseteq \partial M_2 ]
			\leq 2^{v_2} (K'')^{v_2+a} v(P''_1)^{(c^\ast_2 + m) - (v_2 + m)}
			\leq (4K'')^{v_2} v(P''_1)^{c^\ast_2 - v_2}. \]
		We have that $v(P''_1) = m(s_1 - 2 \ell + 1) \geq m \geq n / (5 s_1)$, and together with $c^\ast_2 \leq v_2$ we obtain
		$v(P''_1)^{c^\ast_2 - v_2} \leq (5 s_1)^{v_2} n^{c^\ast_2 - v_2} $.
		This gives
		\[ \probability[ \Eanc \mid \Eval, \Econ ] \leq (20s_1 K'')^{v_2} n^{c^\ast_2 - v_2}, \]
		we conclude by using $20 s_1 K'' \leq K$, which is true by the constant hierarchy.
	\end{proofclaim}

	Thus, by \eqref{equation:keyinequality-correctness} and Claims~\ref{claim:ece1} and \ref{claim:e2}, we obtain that
	\begin{align*}
		\probability[I \subset \partial_2 H]
		 & \leq \probability[\Econ, \Eval] \probability[ \Eanc \mid \Eval, \Econ] \\
		 & \leq K^{v_1} n^{c_1^\ast - v_1} \cdot K^{v_2} n^{c_2^\ast - v_2}       \\
		 & = K^v n^{c-v},
	\end{align*}
	where in the last step we used $v_1 + v_2 = v$ and $c^\ast_1 + c^\ast_2 = c$.
\end{proof}

\section{Proof of \cref{lemma:fromcorrecttospread}}
\label{sec:correct-to-spread-proof}

To prove \cref{lemma:fromcorrecttospread} we use the following simple lemma (this is similar to~\cite[Proposition 2.2]{FKN+21}).
\begin{lemma}\label{lem: bounded degree}
	Let $F$ be a $k$-digraph with $f$ edges and $\Delta(F)\leq d$.
	Then the number of subgraphs of $F$ with $\ell$ edges and $c$ components is at most $(4ekd)^\ell \binom{f}{c}$.
\end{lemma}

\begin{proof}[Proof of \cref{lemma:fromcorrecttospread}]
	We assume that $\lambda \geq d$.
	Our first objective is to show that $\mu$ is $2K^kn^{-1/d}$-spread.
	Let $C\in \cC(n,L)$ be chosen according to $\mu$.
	Let $I$ be a $k$-digraph with vertex set in $[n]$ and with $v$ vertices, $e$ edges and $c$ components.
	We may assume that $I$ is a subgraph of some element of $\cC(n,L)$,
	as otherwise obviously $\probability[I \subset C] = 0$.

	Let $I_1, \dotsc, I_c$ be the components of $I$,
	and for each $i\in [c]$,
	let $\hat{v}_i, \hat{e}_i$ be the number of vertices and edges of each $I_i$, respectively.
	As $C$ is $(d,\lambda)$-balanced, by \ref{balanced:global}, we obtain
	\begin{align*}
		\hat{e}_i\leq d\hat{v}_i.
	\end{align*}

	We call a component of $I$ \emph{large} if $v_i > \frac{n}{2(\ell+r)}$, and \emph{small} otherwise.
	Consequently, as~$C$ is $(d,\lambda)$-balanced and $\lambda\geq d$, \ref{balanced:small} implies that for each small $I_i$,
	\begin{align*}
		\hat{e}_i\leq d\hat{v}_i - d.
	\end{align*}

	Let $v_\ell,e_\ell$ be the total number of vertices and edges in the large components of $I$ and denote by $c_\ell$ the number of large components;
	similarly, we define $v_s,e_s,c_s$ for the small components.
	Summing the above inequalities gives
	\begin{align*}
		e = e_\ell + e_s \leq dv_\ell + d v_s - dc_s = d(v - c_s).
	\end{align*}
	Therefore,
	\begin{align*}
		c - v = c_\ell + c_s - v
		\leq  c_\ell - \frac{e}{d}.
	\end{align*}

	Recall that $\mu$ is $K$-correct.
	Moreover, $e\geq v/k$, since $I$ is a $k$-digraph.
	Hence
	\begin{align}\label{eq:spread}
		\probability[I \subset C]
		\leq  K^v n^{c-v}
		\leq  K^v n^{c_\ell - \frac{e}{d}}
		\leq (K^kn^{-1/d})^e n^{c_\ell}.
	\end{align}

	If $c_\ell > 0$,
	then $v \geq \frac{n}{2(\ell+r)}$ and we must also have $c_\ell \leq 2(\ell+r)$.
	With some foresight, we even bound $n^{k c_\ell}$ from above.
	Using $e \geq v/k$ and that $n$ is large, we have $n^{k c_\ell} \leq n^{2k(\ell+2)} \leq 2^{n/(2k(\ell+r))} \leq 2^e$.
	In fact, $n^{kc_\ell} \leq 2^e$ also holds if $c_\ell=0$.
	All together, this implies that $\probability[I \subset C] 	\leq (2K^kn^{-1/d})^e$.
	Therefore, $\mu$ is $2K^kn^{-1/d}$-spread as desired.
	\medskip

	Now we turn to the second part;
	that is, we assume that $C$ is $(d, \lambda)$-balanced for some $\lambda > d$.
	Here our goal is to show that $\mu$ is $(2K^kn^{-1/d},(e(L)n/r)^{-(\lambda/d - 1)/2}, 8e^2kdK^k)$-strongly-spread.
	So we can finish with $K' = 8e^2kdK^k$.
	We set $\rho=\lambda/d-1>0$ and observe that $\rho\leq k-1$ as a single edge in \ref{balanced:global} testifies.
	Note that all calculations above are still valid.
	We thus employ the same notation as above
	and aim for a stronger estimate as in~\eqref{eq:spread}.

	For all small components $I_i$, we obtain
	$\hat{e}_i\leq d\hat{v}_i - \lambda$
	and hence
	\begin{align*}
		c - v
		 & = c_\ell + c_s - v_\ell - v_s
		\leq c_\ell - \frac{e_s}{d} - \frac{e_\ell}{d} - c_s\left(\frac{\lambda}{d}-1\right)
		= (1+\rho)c_\ell - \frac{e}{d} - \rho c \\
		 & \leq kc_\ell - \frac{e}{d} - \rho c.
	\end{align*}
	As $\mu$ is $K$-correct and since $n^{kc_\ell} \leq 2^e$, this gives
	\begin{align*}
		\probability[I \subset C]
		 & \leq K^v n^{c-v}
		\leq  K^v n^{kc_\ell - \frac{e}{d} - \rho c}
		\leq (2K^kn^{-1/d})^e n^{- \rho c}
		\leq q^e m^{- \rho c/2}
	\end{align*}
	where we set $q=2K^kn^{-1/d}$ and $m=\frac{e(L)}{r}n$.
	Note that $m^{-\rho/2} = (e(L)n/r)^{-(\lambda/d - 1)/2}$.
	Now, let $S \subseteq [n]^k$ which is a subgraph of some $L$-chain that is supported by $\mu$.
	Choose~$j$ with $m^{-\rho/2}|S|\leq j\leq |S|$
	and define $c_I$ for a set $I\subseteq [n]^k$ as the number of components of $I$.
	Then we compute (and also use Lemma~\ref{lem: bounded degree})
	\begin{align*}
		\sum_{I\subseteq S,\, |I|=j} \probability[I \subseteq C]
		 & \leq \sum_{I\subseteq S,\, |I|=j} m^{-\rho c_I/2} q^{|I|}                    \\
		 & = q^j \cdot \sum_{c=1}^j m^{-\rho c_I/2}|\{I\subseteq S\colon |I|=j,c_I=c\}| \\
		 & \leq q^j \cdot \sum_{c=1}^j   m^{-\rho c/2} (4ekd)^j \binom{|S|}{c}          \\
		 & \leq (4ekdq)^j \cdot \sum_{c=1}^j  \frac{(m^{-\rho/2}|S|)^c}{c!}             \\
		 & \leq (4ekdq)^j \cdot e^{m^{-\rho/2}|S|}
		\leq (4e^2kdq)^j,
	\end{align*}
	as desired.
\end{proof}

\section{Threshold Lemma} \label{section:strongspread-to-threshold}

The goal of this section is to prove \Cref{lemma:fromstrongspreadtothreshold}.
Recall that \Cref{theorem:FKNP} comes with an undesired logarithmic factor.
\Cref{lemma:fromstrongspreadtothreshold} in turn describes specific situations for $(d, \lambda)$-balanced chains that allow us to avoid this factor.

Before giving the proof, we briefly explain our strategy.
We will apply a result by Espuny Díaz and Person~\cite{EP23} (similar results were obtained by Spiro~\cite{Spi23}).
Extending the method of Kahn, Narayanan and Park~\cite{KNP21}, their result gives a general framework where the threshold for certain properties can be estimated as in \Cref{theorem:FKNP}, but avoiding the presence of the logarithmic factor.
The assumption in their result is that the corresponding  hypergraph satisfies more restrictive `spread' properties than in \Cref{theorem:FKNP}; and in our case those properties will be deduced from our assumptions and $(d, \lambda)$-balancedness and strong spread.

We connect our definitions with previous work via the notion of multihypergraphs.
Recall that, for a $(k,\ell,r)$-link $L$, the directed hypergraph $\cH(n,L)$ has vertex set $[n]^k$ and edge set $\{E(C)\colon C\in \cC(n,L)\}$.
We have already seen that a probability measure $\mu$ on $\cC(n,L)$ naturally transfers to a measure on the edge set of~$\cH(n,L)$.
Given $\mu$ (which we assume to attain only rational weights) and $\cH(n,L)$,
one can turn $\cH(n,L)$ into a multihypergraph $\cH_{\text{multi}}(n,L,\mu)$
which arises from $\cH(n,L)$ by duplicating or deleting edges of $\cH(n,L)$ such that
for all $S\subseteq V(\cH(n,L))$
\begin{align*}
	\mu(S)= \frac{|\{e\in E(H_{\text{multi}})\colon e=S\}|}{|H_{\text{multi}}|}.
\end{align*}
This point of view allows us to switch between probability measures and multihypergraphs.

Now we translate the definitions of spread to multihypergraphs.
A multihypergraph $H$ is \emph{$q$-spread} if $|\langle I \rangle \cap H | \leq q^{|I|}|H|$ for any $I\subseteq V(H)$, where $\langle I \rangle$ denotes the upset of $I$.
Moreover, $H$ is \emph{$(q,a,b)$-strongly-spread}
if $H$ is $q$-spread and for all $S\subseteq V(G)$ where $S\subseteq e$ for some $e\in E(H)$ and  $a|S|\leq j \leq |S|$, we have
\begin{align*}
	\sum_{I\subseteq S,\, |I|=j}|\langle I \rangle \cap H | \leq (bq)^{j}|H|.
\end{align*}
For $p \in [0,1]$, we denote by  {$H_p^{\textrm{vx-dlt}}$} the random subgraph obtained from $H$ by keeping every vertex with probability $p$.\footnote{We chose this somewhat cumbersome name to highlight the distinction to the random graph $H_p$, which is obtained by keeping edges with probability $p$.}

Given these preliminaries, we can state the following lemma, which is a slightly modified version of a result by Espuny D\'iaz and Person~\cite[Proof of Lemma 2.1]{EP23}.
We say that a hypergraph is \emph{$b$-bounded} if each edge has size at most $b$.

\begin{lemma}\label{thm:strongly-spread}
	For all $K>0$ and $\alpha,\eps\in (0,1)$,
	there exists $C>0$ such that for all sufficiently large $m\in \mathbb{N}$ the following holds.
	Let $H$ be a $b$-bounded multihypergraph on $m$ vertices with $b=\omega(1)$ and $b=O(\sqrt{m})$.
	Assume $H$ is $(q,b^{-\alpha},K)$-strongly-spread with $q\geq 2b/(Cm)$.
	Then for $p\geq Cq$, we have
	\begin{align*}
		\mathbb{P}[H_p^{\text{\upshape{vx-dlt}}} \text{ contains an edge}]\geq 1- \eps.
	\end{align*}
\end{lemma}

\begin{proof}[Proof of \Cref{lemma:fromstrongspreadtothreshold}]
	Note that what we need to show is equivalent to the following: given the assumptions on $L$ and $\mu$, for every $\varepsilon > 0$, there exists $C' = C'(\varepsilon)$ such that for $q \geq C' n^{-1/d}$, a $q$-binomial subset of $[n]^k$ contains an element of $\mathcal{H}(n, L, \mu)$ with probability at least $1 - \varepsilon$.

	Let $\varepsilon > 0$ be arbitrary.
	Let $C$ satisfy the requirements of \cref{thm:strongly-spread} under the inputs $K', \alpha, \varepsilon, b = e(L)n/r$.
	Set $H=\cH_{\text{multi}}(n,L,\mu)$, and note that $H$ has $m = n^k $ vertices.
	Also, $H$ is $b$-bounded with $b=e(L)n/r$, which follows from the definition of closed $L$-chains.
	By our assumptions on $L$-chains we have $e(L) \neq 0$ and $k \geq 2$, so $b = \omega(1)$ and $b = O(\sqrt{m})$ hold.

	Recall that $L$ is non-empty and that sufficiently large closed $L$-chains are $(d, \lambda)$-balanced.
	We claim that $d > 1/(k-1)$.
	Indeed, consider a subgraph $I$ of an $L$-chain consisting of exactly one edge of size $k$.
	Then $I$ has $k$ vertices and one edge. By $(d, \lambda)$-balancedness, we get that $1 \leq d k - \lambda < d (k-1)$, which implies the claim.

	Moreover, $H$ is $(K' n^{-1/d},(e(L)n/r)^{-\alpha},K')$-strongly-spread.
	Note that $K' n^{-1/d} \geq 2(e(L)n/r)/(Cn^k) = 2b/(Cm)$, where we used that $d \geq 1/(k-1)$ in the first inequality.
	Hence the conclusion (with $C' = C K'$) follows from \cref{thm:strongly-spread}.
\end{proof}

\section{Conclusion}
\label{sec:conclusion}

In this paper, we investigated the robustness of Hamiltonicity in hypergraphs under random sparsifications.
Our main result, \cref{thm:main}, states that Hamiltonicity is inherited to a random sparsification of a host graph $G$, provided that $G$ robustly exhibits Hamilton connectedness at the local level.
As a consequence, we obtain (random) robustness versions of many classic theorems and several recent breakthroughs.
Importantly, our method works in settings (such as $d$-degree conditions for $d < k-1$) where the host graph is not necessarily connected, which was an obstacle in previous works.

Our work also implies counting versions of these results, which can be easily derived from the existence of spread distributions (see, for example \cite[Section 1.4.1]{KMP23} for such a deduction).
For instance, this shows that there are at least $\exp(n \log n - O(n))$ Hamilton $\ell$-cycles under the minimum $d$-degree conditions of \cref{thm:tight-hamilton-cycles-codegree,thm:tight-hamilton-cycles-k-2-degree,thm:loose-hamilton-cycles-codegree,thm:loose-hamilton-cycles-l=1}.
This recovers and extends the work of Glock, Gould, Joos, Kühn and Osthus~\cite{GGJKO2020}, Montgomery and Pavez-Signé~\cite{MP2023} and Kelly, Müyesser and Pokrovskiy~\cite{KMP23}.
For a more detailed discussion of the subject, we refer the reader to the work of Montgomery and Pavez-Signé~\cite{MP2023}.

The applications in \cref{sec:applications} illustrate the broad range of scenarios in which \cref{thm:main} can be used to establish robustness for Hamiltonicity.
The list of outcomes is by no means exhaustive and we believe that many further applications are possible.
For instance, it is not hard to deduce that Hamiltonicity is preserved in sparsifications of robust expanders~\cite{KOT2010}, loose cycles in quasirandom hypergraphs~\cite{LMM16} or powers of tight cycles in hypergraphs~\cite{PSS23}.
{Our outcomes do not take advantage of the fact that \cref{thm:main} is formulated in terms of directed hypergraphs.
An illustrative application for this would be a robust version of the Dirac-type theorem for oriented graphs due to Keevash, Kühn and Osthus~\cite{KKK09}, which can be derived from \cref{thm:main} in a similar way as the results of~\cref{sec:loose-cycles-proofs}}.

Key to all of these applications is the notion of Hamilton connectedness.
Abstractly speaking, when it comes to using Hamiltonicity (in the broad sense of this paper) as an input condition, then Hamilton connectedness appears to be the `right' form of it.
In related work~\cite{LS24a}, we provide a structural decomposition of this property together with a series of new applications, which then thanks to \cref{thm:main} receive an immediate `upgrade' to the random robustness setting.

Looking ahead, multiple open questions come to mind.
Our main result is formulated in terms of coarse thresholds.
We believe that this can be improved to semi-sharp thresholds as in the work of Kahn, Narayanan and Park~\cite{KNP21}.
Beyond quantitative considerations, it would be quite interesting to know whether similar results are possible in other settings that related random graphs with deterministic conditions.
For instance, is it possible to formulate and prove an analogue of \cref{thm:main} for the (random) resilience setting, which was popularised Sudakov and Vu~\cite{SV08}?
Or can we capture some of the progress in the field of random perturbations, initiated by Bohman, Frieze and Martin~\cite{BFM03}, in a similar way?

Apart from these `meta-questions', there are also many interesting concrete problems.
Among the most natural is the question how robustness relates to decomposition problems.
For instance, it is known that the minimum degree threshold $\mu$ for the existence of a triangle-decomposition of a graph is somewhere between $0.75$ and $0.83$~\cite{DP21}.
Provided that the trivial divisibility conditions are met, does a random sparsification of an $n$-vertex graph $G$ with $\delta(G) \geq (\mu n + o(1))n$ still contain a triangle-decomposition?

\bibliographystyle{amsplain}
\bibliography{bibliography}

\end{document}